\documentclass[12pt]{amsart}
\usepackage[margin=1in]{geometry}
\usepackage[T1]{fontenc}
\usepackage{lmodern}
\usepackage{amsmath,amssymb,amsthm,mathtools,microtype}
\usepackage{enumitem}
\usepackage[hidelinks,hypertexnames=false]{hyperref}
\newtheorem{theorem}{Theorem}[section]
\newtheorem{classical}{Theorem}

\newtheorem{lemma}[theorem]{Lemma}
\newtheorem{corollary}[theorem]{Corollary}
\newtheorem{proposition}[theorem]{Proposition}
\theoremstyle{definition}
\newtheorem{definition}{Definition}[section]
\newtheorem{example}[theorem]{Example}
\theoremstyle{remark}
\newtheorem{remark}[theorem]{Remark}
\newcommand{\R}{\mathbb R}
\newcommand{\Sym}{\operatorname{Sym}}
\newcommand{\tr}{\operatorname{tr}}
\newcommand{\diag}{\operatorname{diag}}
\newcommand{\II}{\mathrm{II}}
\newcommand{\Pplus}{\mathcal P^+_{\lambda,\Lambda}}
\newcommand{\Pminus}{\mathcal P^-_{\lambda,\Lambda}}
\newcommand{\norm}[1]{\lVert#1\rVert}

\title[Nirenberg's Theorem and Traceless Symmetric Matrices]{Nirenberg's Theorem and Singular Solutions from Traceless Symmetric Matrices}

\author[B.Z. Chu]{BaoZhi Chu}
\address[B.Z. Chu]{Department of Mathematics, Northwestern University,
2033 Sheridan Road, Evanston, IL 60208}
\email{bzchu@northwestern.edu}

\begin{document}

\begin{abstract}
    A classical theorem of Nirenberg gives interior $C^{2,\alpha}$ regularity for continuous viscosity solutions of fully nonlinear uniformly elliptic equations $F(D^2u)=0$ in dimension two. In dimensions five and higher, the work of Nadirashvili–Tkachev–Vlăduţ shows that such solutions need not be $C^2$. In this paper, we construct a family of $C^{1,1}\setminus C^2$ functions $\{w_m\}_{m\geq 3}$ and uniformly elliptic operators $F_m$ such that $F_m(D^2 w_m)=0$ in 
    $\mathbb{R}^{d_m}$, where
    $d_m=m(m+1)/2-1$. These solutions are given by $\operatorname{tr}(X^3)/\sqrt{\operatorname{tr}(X^2)}$ on the space of traceless real symmetric $m\times m$ matrices. When $m=3,4$, this family recovers the known singular solutions associated with the five-dimensional Cartan isoparametric cubic and the nine-dimensional Hsiang minimal cubic, respectively. Our construction treats all $m\geq3$ in a unified way by analyzing the
Hessian directly on the space of traceless symmetric matrices.
    
We also prove a geometric criterion for pointwise $C^{2,\beta}$
regularity of continuous viscosity solutions $u$ of uniformly elliptic
equations $F(D^2u)=0$ in every dimension $n\geq3$.
Suppose that $n-2$ smooth hypersurfaces pass through a point $p$, each carrying a smooth vector field along which the directional derivative of $u$ is constant on that hypersurface. If these vector fields are linearly independent at $p$, then $u$ is twice differentiable at $p$ and admits a pointwise $C^{2,\beta}$ expansion for some $\beta\in (0,1)$.
Our criterion is quantitative and yields interior $C^{2,\beta}$
regularity under uniform geometric hypotheses.
As an application, we extend Nirenberg's theorem to a general
class of solutions in dimensions three and higher.
As a consequence, we establish interior $C^{2,\beta}$ regularity
for solutions of Dirichlet problems with certain group symmetries.
This extends a result of
Nadirashvili--Vl\u{a}du\c{t} for axially symmetric problems.
\end{abstract}

\maketitle

\section{Introduction}
We denote by $\Sym(n)$ the space of real
symmetric $n\times n$ matrices. 
Let $F:\Sym(n)\to \R$ be a \emph{uniformly elliptic} operator. Namely, there exist $\Lambda\ge \lambda>0$ such that
\begin{equation}\label{eq:ellipticity}
\lambda\tr N\leq F(M+N)-F(M)\leq\Lambda\tr N\quad
\text{for every } M\in\Sym(n)\text{ and every }N\geq 0. 
\end{equation}

A classical theorem of
Nirenberg in the 1950s gives interior $C^{2,\alpha}$ regularity for solutions of
uniformly elliptic equations in two dimensions. We begin with its
modern form. Throughout the paper, we use $B_r(p)$ for a
Euclidean ball centered at $p$ of radius $r$ and write $B_r=B_r(0)$. 

\begin{classical}[Nirenberg \cite{Nirenberg}]\label{thm:nirenberg}
Let $F:\Sym(2)\to\R$ satisfy \eqref{eq:ellipticity}, and let
$u$ be a continuous viscosity solution of $F(D^2u)=0$ in
$B_1\subset\R^2$. There exist positive constants $\alpha$ and $C$,
depending only on $\lambda,\Lambda$, such that $u\in C^{2,\alpha}_{\mathrm{loc}}(B_1)$ and
$\norm{u}_{C^{2,\alpha}(B_{1/2})}
 \leq C\norm{u}_{L^\infty(B_1)}$.
\end{classical}
Without additional assumptions, Nirenberg's theorem cannot be simply extended to general higher dimensions.
This is shown by the work of
Nadirashvili--Tkachev--Vl\u{a}du\c{t} \cite{NTV, NTV-book}, which constructs a homogeneous $C^{1,1}\setminus C^2$ solution with an isolated singularity in dimension five. 
If we assume additionally that $F$ is convex or concave, then the Evans--Krylov theorem \cite{Evans-82, Krylov-82, Krylov-83} gives $u\in C^{2,\alpha}_{\mathrm{loc}}$ in any dimension; see also \cite{Caffarelli1995FullyNE, CS, GT}. Some extensions of the Evans--Krylov theorem 
are established by Cabr\'e--Caffarelli \cite{CC-jmpa} and
Caffarelli--Yuan \cite{CY}.

\subsection{Singular solutions from traceless symmetric matrices}\label{sec-1-1}
The first result in this paper extends the singular solution in \cite{NTV} to a family of singular solutions $\{w_m\}$ in dimensions
$$d_m:=m(m+1)/2-1,\quad m= 3,4,\dots.$$
The first few dimensions in this family are $5,9,14,20,27,35,\ldots$.

Let $\Sym_0(m)$ be the space of traceless real symmetric $m\times m$ matrices, endowed with the Frobenius inner product $\langle X,Y\rangle:=\tr(XY)$
and its norm
$\norm X_F:=\sqrt{\tr(X^2)}$.
Clearly, $\dim\Sym_0(m)=d_m$.
Define a linear isometry $\iota_m: \R^{d_m}\to \Sym_0(m)$ by 
$$ \textstyle
 \iota_m(x):=\sum_{i=1}^{d_m}x_i E_i,\quad x\in \R^{d_m},
$$
where 
$E_1,\dots,E_{d_m}$ is a fixed orthonormal basis of $\Sym_0(m)$. 
Consider
$$
 W_m(X):=
 {\tr(X^3)}/{\norm X_F},
 \quad X\in\Sym_0(m),
$$
where $W_m(0)=0$. Our first main result is as follows. 
\begin{theorem}\label{thm:trace-cubic}
For every $m\geq3$, there exists
$F_m:\Sym(d_m)\to\R$ satisfying \eqref{eq:ellipticity} with
$\lambda=1$ and $\Lambda=1+(m+3)(d_m+3)$
such that $w_m:=W_m\circ\iota_m$ is a viscosity solution of
$$
F_m(D^2w_m)=0
\quad\text{in }\R^{d_m}.
$$
The solution $w_m$ belongs to $C^{1,1}(\R^{d_m})$,
is homogeneous of degree two, and is real analytic in
$\R^{d_m}\setminus\{0\}$, but is not twice differentiable
at $0$.
\end{theorem}

Note that the ellipticity constants of $F_m$
are independent
of the choice of orthonormal basis of $\Sym_0(m)$, though the coordinate forms of $F_m$ and $w_m$
differ by an orthogonal transform. 

Clearly, Theorem \ref{thm:trace-cubic} cannot be extended to $m=2$,
since $d_2=2$ and its non-$C^2$ conclusion would contradict
Nirenberg's theorem. Indeed, $\tr(X^3)=0$ for every $X\in\Sym_0(2)$, so $W_2\equiv0$.

Theorem \ref{thm:trace-cubic} in the cases $m=3,4$ recovers known singular solutions whose defining cubics have classical geometric interpretations, as described below. 
To the best of our knowledge, the realization of the functions
$w_m$ as viscosity solutions of uniformly elliptic fully nonlinear
equations was not previously known for $m\geq5$.
We are not aware of an analogous geometric interpretation of their
defining cubics $\tr(X^3)$ for $m\geq 5$.

When $m=3$, after a suitable choice of orthonormal basis,
the cubic $\tr(X^3)$ on $\Sym_0(3)$ is a nonzero constant
multiple of the five-dimensional \emph{Cartan isoparametric cubic} \cite[p.190]{Cartan}
$$
 P(x):=x_1^3
 +\frac{3x_1}{2}(x_3^2+x_4^2-2x_5^2-2x_2^2)
 +\frac{3\sqrt3}{2}
 (x_2x_3^2-x_2x_4^2+2x_3x_4x_5).
$$
Consequently, up to an orthogonal change of variables and
multiplication by a nonzero constant, $w_3$ coincides with
the singular solution $P(x)/|x|$ constructed by
Nadirashvili--Tkachev--Vl\u{a}du\c{t}~\cite{NTV}.
The regular level sets of the Cartan cubic $P$ on $\mathbb S^4$
give, up to isometries of $\mathbb S^4$, all complete connected isoparametric
hypersurfaces with three distinct principal curvatures;
see \cite{Cartan}. The identification of $P$ with $\tr(X^3)$ and its relation to Jordan algebras are
discussed by Tkachev~\cite{Tkachev2014,Tkachev2019};
see also Appendix~\ref{app-sec-realizations}.

When $m=4$, the cubic $\tr(X^3)$ on $\Sym_0(4)$ is exactly $3$ times the nine-dimensional \emph{Hsiang minimal cubic}. Its zero level set is a minimal hypercone in $\R^9$; see \cite{Hsiang1967}.
In a suitable orthonormal basis,
the Hsiang cubic takes the coordinate form
$$
D(x)\coloneqq \det 
\begin{pmatrix}
    x_1 & x_2 & x_3 \\
    x_4 & x_5 & x_6 \\
    x_7 & x_8 & x_9 
\end{pmatrix}
$$
as observed in \cite{Tkachev2010}; see also
Appendix \ref{app-sec-realizations}.
Therefore, up to an orthogonal change of variables and
multiplication by a nonzero constant, $w_4$ coincides with
the known singular solution $D(x)/|x|$ in dimension nine.
Its realization as a viscosity solution of a uniformly elliptic fully
nonlinear equation is due to C. K. Smart (unpublished), as
recorded in
\cite[Remark~4.1.5]{NTV-book} and \cite[p.1008]{S}.

Our proof of Theorem \ref{thm:trace-cubic} is unified for all $m\ge 3$ by differentiating $W_m$
directly on $\Sym_0(m)$ and analyzing its Hessian there.
The basis-independent formula for $W_m$ allows us to avoid
lengthy coordinate calculations for the Hessian of $w_m$
on $\R^{d_m}$.
One main step is to prove
$$
(\Lambda-1)\lambda_{\max}\bigl(D^2W_m(X)-D^2W_m(Y)\bigr)
\geq
\left|\operatorname{Tr}\bigl(D^2W_m(X)-D^2W_m(Y)\bigr)\right|
$$
for all $X,Y\in\Sym_0(m)\setminus\{0\}$.
Here the Hessians are identified with self-adjoint operators
on $\Sym_0(m)$ through the Frobenius inner product, and
$\lambda_{\max}$ and $\operatorname{Tr}$ denote their largest
eigenvalue and operator trace, respectively.
We prove this estimate using elementary matrix inequalities,
with several estimates refined to obtain the stated explicit
value of $\Lambda$.
Once this estimate is established, the remaining steps, including the construction of $F_m$, are similar to those in \cite{Chu}.

\subsection{An extension of Nirenberg's theorem}
Consider, in dimensions $n\geq3$, the equation
\begin{equation}\label{eq:equation}
 F(D^2u)=0\quad\hbox{in }B_1\subset\R^n,
\end{equation}
where $F$ is only assumed to be uniformly elliptic.
It is well-known that
every continuous viscosity solution $u$ of \eqref{eq:equation} is $C^{1,\alpha}_{\mathrm{loc}}$ for some $\alpha>0$; see \cite{Caffarelli1995FullyNE, Trudinger-89}. 

The second aim of this paper is to give a geometric criterion for
the existence of a pointwise $C^{2,\beta}$ expansion of a solution $u$. Under
uniform geometric hypotheses, this criterion yields interior $C^{2,\beta}$ regularity.
To guarantee such an expansion at a point $p$, our criterion
requires $n-2$ hypersurfaces through $p$, each carrying a vector
field along which the directional derivative of $u$ is constant
on that hypersurface. These vector fields are
assumed to be linearly independent at $p$.

Our criterion is quantitative, so we first introduce some geometric
parameters. Throughout this paper, a $C^2$ \emph{hypersurface}
in $\R^n$ is a subset that, near each of its points and after
a rigid motion, is the graph of a $C^2$ function on an open
subset of $\R^{n-1}$.

\begin{definition}\label{def:admissible}
Let $S\subset \R^n$ be a $C^2$ hypersurface with
$p\in S$. A radius $\rho>0$ is \emph{admissible} for
$S$ at $p$ if $S\cap B_\rho(p)$ is relatively
closed in $B_\rho(p)$ and
$$ 
 \rho 
 \sup_{S \cap B_\rho(p)}\norm{\II_{S}}\leq 1.
$$
\end{definition}
Here $\II_S$ is the Euclidean second fundamental form of $S$;
its norm is the largest absolute principal curvature.
Clearly, an admissible radius $\rho$ always exists. An admissible radius is not unique; if $\rho$ is admissible, then every $\rho'\in (0,\rho)$ is also admissible. If $S$ is an affine
hyperplane, then every $\rho>0$ is admissible.

\begin{definition}\label{def:independence}
For unit vectors $\mu_1,\dots,\mu_{n-2}\in\R^n$, we define their
\emph{independence number} by
$$
\textstyle
 \delta(\mu_1,\dots,\mu_{n-2})
 \coloneqq\min \big\{               \left|\sum_{i=1}^{n-2}a_i \mu_i\right| :   a\in\R^{n-2},~|a|=1 \big\}.
$$
\end{definition}
This number lies in $[0,1]$ and is positive precisely when
the vectors are linearly independent. It equals $1$ when the vectors are
orthonormal.

Let $S$ be a $C^2$ hypersurface through point $p$.
For a $C^{1,\gamma}$ vector field $\nu$ along $S$, with
$0<\gamma\leq1$, we measure its variation in $B_\rho (p)$ by
\begin{equation}\label{eq:field-bound}
 \rho\norm{D_S\nu}_{L^\infty(S\cap B_\rho(p))}
 +\rho^{1+\gamma}[D_S\nu]_{\gamma;S\cap B_\rho(p)}\leq K.
\end{equation}
Here $D_S\nu$ is the derivative along $S$, and the bracket denotes
its H\"older (Lipschitz) seminorm; see \S\ref{sec:main-proof} for precise definitions.
If $\nu$ is 
constant, then \eqref{eq:field-bound} holds for $K=0$ and
$\gamma=1$.

\begin{theorem}\label{thm:main}
Let $n\geq3$, $0<\gamma\leq 1$, $K\geq 0$, and let
$F:\Sym(n)\to\R$ satisfy \eqref{eq:ellipticity}. Suppose $u$
is a continuous viscosity solution of \eqref{eq:equation}.
 For each $1\le i\le n-2$, let
$S_i\subset \R^n$ be a $C^2$ hypersurface 
with admissible radius $\rho\leq1/8$ at $0\in S_i$, and let
$\nu_i\in C^{1,\gamma}(S_i\cap B_\rho;\R^n)$ satisfy
\eqref{eq:field-bound} with $p=0$, $S=S_i$, and $\nu=\nu_i$.
Assume that $|\nu_i(0)|=1$, that
$\delta:=\delta(\nu_1(0),\dots,\nu_{n-2}(0))>0$, and that
$$
 \nu_i(x)\cdot Du(x)=\text{constant}
 \quad\text{on }S_i\cap B_\rho
 \quad\text{for every } 1\le i \le n-2.
$$
Then $Du$ is differentiable at $0$, 
\begin{equation}\label{eq:main-coefficients}
 |Du(0)|+\norm{D^2u(0)}\leq C\norm{u}_{L^\infty(B_1)}, \quad \text{and}
\end{equation}
\begin{equation}\label{eq:main-expansion}
 \left|u(x)-u(0)-Du(0)\cdot x- 2^{-1}x^TD^2u(0)x\right|
 \leq C \norm{u}_{L^\infty(B_1)}|x|^{2+\beta},
 \quad x\in B_{c\rho}.
\end{equation}
Here $\beta\in(0,\gamma)$ and $c>0$ depend only on
$n,\lambda,\Lambda,\gamma$; $C$ additionally depends on an upper
bound for $K$ and positive lower bounds for $\rho$ and $\delta$.
\end{theorem}

The vector field $\nu_i$ need
not be tangent or normal to the hypersurface $S_i$. The hypersurfaces
need not meet transversely and may even coincide.
For constant vector fields $\{\nu_i\}$, take $\gamma=1$ and $K=0$. Then $\beta,c$
depend only on $n,\lambda,\Lambda$, and $C$ depends additionally
only on positive lower bounds for $\rho$ and $\delta$.

\begin{remark}\label{rem:parameters}
The dependence of $C$ on $\rho$ and $\delta$ in
Theorem \ref{thm:main} cannot be dropped: the estimate
\eqref{eq:main-expansion} may fail if either dependence is removed; see Example \ref{ex:dependence} for counterexamples
in every dimension $d_m+1$ with $m\ge 3$.
These counterexamples utilize the singular solutions $w_m$ constructed in Theorem \ref{thm:trace-cubic}.
\end{remark}

Trudinger \cite{Trudinger-twice-differentiability} proved that every
continuous viscosity solution of \eqref{eq:equation} is twice
differentiable almost everywhere. Recently, Chu \cite{Chu}
quantitatively strengthened this result: for each $0\leq\alpha<1$,
the set where $u$ has no pointwise $C^{2,\alpha}$ expansion\footnote{In \cite{Chu}, pointwise $C^{2,\alpha}$ expansion refers to little-$o$ remainder. More precisely, a continuous
function $u$ has a $C^{2,\alpha}$ expansion at $x$ if there
exists a polynomial $q_x$ of degree at most two such that
$u(y)=q_x(y)+o(|y-x|^{2+\alpha})$ as $y\to x$.} has
Hausdorff dimension at most $n-\varepsilon_0(1-\alpha)$, where
$\varepsilon_0>0$ depends only on $n,\lambda,\Lambda$.
For $\alpha=0$, this is precisely the set where $u$ is not twice
differentiable. 
Determining whether $u$ is twice differentiable
or admits a $C^{2,\alpha}$ expansion at a point
remains a delicate issue. 
Theorem \ref{thm:main} provides a geometric criterion
ensuring twice differentiability and a pointwise $C^{2,\alpha}$
expansion for some $\alpha>0$.
In dimension three, it implies in particular that $D^2u(0)$
exists whenever $\partial_eu$ is constant on a $C^2$ surface
containing $0$, for some constant unit vector $e$.

For viscosity solutions $u$ of \eqref{eq:equation},
pointwise twice differentiability and the existence of pointwise
$C^{2,\alpha}$ expansions must be distinguished from local $C^2$
and $C^{2,\alpha}$ regularity. Indeed, in dimensions $n\geq5$,
the results in \cite{Chu} show that the singular set, consisting
of points at which $u$ has no $C^2$ neighborhood, can have
positive Lebesgue measure, even though $u$ admits pointwise
$C^{2,\alpha}$ expansions almost everywhere for every
$0\leq\alpha<1$. However, if we assume additionally that
$F\in C^1$ and $DF$ is uniformly continuous on $\Sym(n)$, then
the partial regularity theorem of Armstrong--Silvestre--Smart
\cite{ASC} gives a relatively closed set $\Sigma\subset B_1$
of Hausdorff dimension at most $n-\varepsilon$, where
$\varepsilon>0$ depends only on $n,\lambda,\Lambda$, such that
$u\in C^{2,\beta}_{\mathrm{loc}}(B_1\setminus\Sigma)$ for every $0<\beta<1$.

As a consequence of Theorem \ref{thm:main}, we obtain the following
interior $C^{2,\beta}$ regularity result under uniform geometric
hypotheses.

\begin{corollary}\label{cor:uniform}
Let $n\geq 3$, $0<\gamma\leq 1$, $K\geq 0$, and let $F:\Sym(n)\to\R$ satisfy \eqref{eq:ellipticity}.
Suppose $u$ is a continuous viscosity solution of \eqref{eq:equation}.
For every $1\le i\le n-2$ and  $p\in B_{1/2}$, let $S_{i,p}\subset \R^n$ be a $C^2$ hypersurface
with admissible radius $\rho\le 1/8$ at $p\in S_{i,p}$, 
and let
$\nu_{i,p}\in C^{1,\gamma}(S_{i,p}\cap B_\rho(p);\R^n)$ satisfy
\eqref{eq:field-bound} 
with $S=S_{i,p}$ and $\nu=\nu_{i,p}$. 
Assume for every $p\in B_{1/2}$ that
$|\nu_{i,p}(p)|=1$, that $\delta(\nu_{1,p}(p),\dots,\nu_{n-2,p}(p))\geq\delta>0$, and that
$$
 \nu_{i,p}(x)\cdot Du(x)=\text{constant}
 \quad\text{on }S_{i,p}\cap B_\rho(p)
 \quad \text{for every }1\le i\le n-2.
$$
Then $u\in C^{2,\beta}(B_{1/2})$ and
\begin{equation}\label{eq:uniform-estimate}
 \norm{u}_{C^{2,\beta}(B_{1/2})}
 \leq C\norm{u}_{L^\infty(B_1)}.
\end{equation}
Here $\beta\in(0,\gamma)$ depends only on $n,\lambda,\Lambda,\gamma$;
$C$ additionally depends on an upper bound for $K$ and positive
lower bounds for $\rho$ and $\delta$.
\end{corollary}

The hypersurfaces and vector fields may be chosen separately at each base
point $p\in B_{1/2}$; only the parameters $\rho,\delta,\gamma, K$ must be uniform.
Again for constant vector fields $\{\nu_{i,p}\}$, one may take $\gamma=1$ and $K=0$.
Then $\beta$
depends only on $n,\lambda,\Lambda$, and $C$ depends additionally
only on positive lower bounds for $\rho$ and $\delta$.

To present two applications of Corollary \ref{cor:uniform}, we consider the following decomposition
\begin{equation}\label{eq:blocks}
 \R^n=\R^d\times\R^m\times
 \R^{r_1}\times\cdots\times\R^{r_k},
 \quad d+k\leq2,
\end{equation}
where $d,m,k\geq0$, $r_i\geq1$, and
$d+m+\sum_i r_i=n$. Accordingly, we write $x\in \R^n$ by $x=(z,y,x^{(1)},\dots,x^{(k)})$.  Here components with zero dimension should be omitted. 

\begin{corollary}\label{cor:quadratic}
Let $n\geq3$, $0<\gamma\leq1$, and $L\geq1$.
Let $F:\Sym(n)\to\R$ satisfy \eqref{eq:ellipticity}, and let $u$
be a continuous viscosity solution of \eqref{eq:equation}.
Suppose that $\Phi\in C^{2,\gamma}(B_{3/4};\R^n)$ satisfies $\det D\Phi\ne 0$ in $B_{3/4}$, and
$$
 \norm{D\Phi}_{C^{1,\gamma}(B_{3/4})}
 +\norm{(D\Phi)^{-1}}_{L^\infty(B_{3/4})}\leq L.
$$
Using the decomposition \eqref{eq:blocks}, write
$\Phi=(Z,Y,X^{(1)},\dots,X^{(k)})$.
Assume that $u$ takes the form
$$
 u(x)=U\bigl(Z(x),Q_1(X^{(1)}(x)),\dots,Q_k(X^{(k)}(x))\bigr)\quad \text{for }x\in B_{3/4},
$$
where each $Q_i$ is a real polynomial of degree at most two and
$U$ is a function on the set of attained arguments.
Then $u\in C^{2,\beta}(B_{1/2})$ and
\eqref{eq:uniform-estimate} holds, with $\beta\in(0,\gamma)$ depending only on
$n,\lambda,\Lambda,\gamma$ and $C$ additionally depending on an upper bound for $L$.
\end{corollary}

Note that constants $\beta$ and $C$ are independent of $Q_i$.
The map $\Phi$ need not be globally one-to-one in $B_{3/4}$, and no regularity
of $U$ is assumed beyond the continuity of $u$.
When $\Phi$ is the identity map, taking $\gamma=1$ and $L=2$
gives $\beta$ and $C$ depending only on $n,\lambda,\Lambda$.

We provide several illustrating examples by
taking $\Phi$ to be the identity map. 
Any solution of the form $u(x)=U(x_1,x_2)$ satisfies the condition in
Corollary \ref{cor:quadratic}, so this corollary may be viewed as an
extension of Nirenberg's theorem to dimensions three and higher.
Other examples include solutions of the forms
$$
U(x_1,x_2^2+\cdots+x_n^2)
\quad\text{and}\quad
U(|x'|^2,|x''|^2),
$$
where $x'=(x_1,\dots,x_\ell)$ and
$x''=(x_{\ell+1},\dots,x_n)$.
Further examples, involving more general quadratic expressions, include
$$
U(x_1^2+x_2-x_3-x_4^2),\quad
U(x_1,x_2x_3),\quad
U(x_1^2-x_2^2,x_3^2-x_4^2),\quad \text{and many others}.
$$

The next corollary concerns the following Dirichlet problem:
\begin{equation}\label{eq:dirichlet}
 F(D^2u)=0\quad\hbox{in }\Omega,\qquad
 u=g\quad\hbox{on }\partial\Omega,
\end{equation}
where $\Omega\subset\R^n$ is a bounded $C^2$ domain, $g\in C^0(\partial\Omega)$, and $F:\Sym(n)\to\R$ satisfies \eqref{eq:ellipticity}. It is well-known that there exists a unique viscosity solution $u\in C^0(\overline\Omega)$ of \eqref{eq:dirichlet}; see \cite{Ishii-89, Jensen-88}.
We further impose group symmetry assumptions on $(\Omega,F,g)$. 
Take $m=0$ in \eqref{eq:blocks}. For each $i$, let
$\mathcal G_i\subset O(r_i)$ be a subgroup acting transitively on
the unit sphere in $\R^{r_i}$, and let
$$
 \mathcal G=\mathcal G_1\times\cdots\times\mathcal G_k
$$
act on $(x^{(1)},\dots,x^{(k)})$ and leave the $z$ variable fixed.
For example, one may take $\mathcal G_i=O(r_i)$, or
$\mathcal G_i=SO(r_i)$ when $r_i\geq2$.
Assume the following invariances:
\begin{equation}\label{domain-G-invariant}
     O\Omega=\Omega \quad \text{for every }O\in\mathcal G.
\end{equation}
\begin{equation}\label{F-G-invariant}
   F(O^TMO)=F(M) \quad \text{for every } M\in \Sym(n)\text{ and } O\in\mathcal G. 
\end{equation}
\begin{equation}\label{g-G-invariant}
    g(O \cdot )=g(\cdot) \text{ on }  \partial\Omega\quad \text{for every } O\in \mathcal G.
\end{equation}

\begin{corollary}\label{cor:dirichlet}
For $n\geq3$, let the group $\mathcal G$
be as above. Let $\Omega\subset\R^n$ be a bounded $C^2$ domain satisfying \eqref{domain-G-invariant}.
Suppose that $F:\Sym(n)\to\R$ satisfies \eqref{eq:ellipticity} and \eqref{F-G-invariant}, and $g\in C^0(\partial\Omega)$ satisfies \eqref{g-G-invariant}. Then there exists $\beta\in (0,1)$, depending only on $n,\lambda,\Lambda$, such that
the unique continuous viscosity solution $u$ of
\eqref{eq:dirichlet} satisfies $u\in C^{2,\beta}_{\mathrm{loc}}(\Omega)$ and
\begin{equation}\label{eq:dirichlet-estimate}
 \norm{u}_{C^{2,\beta}(\Omega')}
 \leq C\bigl(\norm{g}_{L^\infty(\partial\Omega)}+|F(0)|\bigr)
\end{equation}
for every open set $\Omega'\Subset\Omega$,
where $C$ depends on 
$n,\lambda,\Lambda, \operatorname{diam}\Omega$, and
$\operatorname{dist}(\Omega',\partial\Omega)$.
\end{corollary}

In the axially symmetric case, i.e. $\mathcal G=O(n-1)$, Corollary \ref{cor:dirichlet} under stronger technical hypotheses  was known by
Nadirashvili--Vl\u{a}du\c{t}; see \cite[Theorem~1]{NV} for $n=3$ and its following remark for general $n$.

Let us briefly describe the proof of Theorem \ref{thm:main}.
 We first use a Krylov-type boundary estimate,
in the form established by
Silvestre--Sirakov \cite{SS}, together with a Caffarelli-type iteration argument, to obtain pointwise
$C^{1,\eta}$ expansions at $0$ for directional derivatives
in $n-2$ independent directions.
Integrating these expansions shows that, after subtraction
of a quadratic polynomial, the solution differs from a
function of the remaining two variables by at most
$C|x|^{2+\eta}$ near $0$.
We then combine Caffarelli's approximation method with
Nirenberg's theorem to obtain the desired quadratic expansion.

The paper is organized as follows. Section \ref{sec:matching}
recalls a boundary estimate needed in the proof of Theorem \ref{thm:main}. Section~\ref{sec:criterion} proves Theorem~\ref{thm:main}
and Corollary~\ref{cor:uniform}.  Section~\ref{sec:applications}
proves Corollary \ref{cor:quadratic} and \ref{cor:dirichlet}. Section \ref{sec:examples} establishes
Theorem \ref{thm:trace-cubic} and the counterexamples in
Remark \ref{rem:parameters}. Appendix~\ref{app-sec-realizations}
discusses Cartan and Hsiang cubics.

\section{Preliminary: Krylov-type boundary estimate}\label{sec:matching}
Recall the definition of \emph{Pucci's extremal operators}.
For a symmetric matrix $M$ with eigenvalues $\mu_j$ and $\Lambda\ge \lambda>0$, define
$$
\textstyle
 \Pplus(M):=\Lambda\sum_{\mu_j>0}\mu_j+
  \lambda\sum_{\mu_j<0}\mu_j\quad \text{and}\quad
 \Pminus(M) :=\lambda\sum_{\mu_j>0}\mu_j+
                 \Lambda\sum_{\mu_j<0}\mu_j.
$$
In the following,
we write $B_r^+ :=B_r\cap\{x_n>0\}$ and $B_r' :=B_r\cap\{x_n=0\}$.

\begin{proposition}[\cite{SS}]\label{lem:flat-expansion}
Let $w\in C^0(B_1^+\cup B_1')$ be a bounded viscosity solution of 
\begin{equation}\label{eq:Pucci-class}
 \Pminus(D^2w)\leq0\leq\Pplus(D^2w)
\end{equation}
in $B_1^+$.
Suppose that $w=0$ on $B_1'$. 
Then there exists $a\in\R$ such that for all $x\in B_{1/4}^+$,
$$
 |w(x)-ax_n|\leq C\norm w_{L^\infty(B_1^+)}|x|^{1+\alpha_0}\quad \text{and}\quad  |a|\leq C\norm w_{L^\infty(B_1^+)},
$$
where positive constants $\alpha_0$ and $C$
depend only on $n,\lambda,\Lambda$.
\end{proposition}

This is a Krylov-type boundary estimate \cite{Krylov-83} in the form of Silvestre--Sirakov \cite[Lemma~3.1]{SS} with $K=L=0$ there. We sketch a proof below using Caffarelli's simplification of Krylov's original proof; see \cite[\S IV.3]{Kazdan} for an exposition.

\begin{proof}
If $w\equiv 0$, the conclusion is immediate. Otherwise,
replacing $w$ by $w/\norm{w}_{L^\infty(B_1^+)}$, we may assume $\norm w_{L^\infty(B_1^+)}=1$.  Choose $b>0$ with
$\lambda(b+1)>\Lambda(n-1)$. Then a computation gives
$$
 \Pminus(D^2|x-y|^{-b})>0 \quad \text{for all }x,y\in \R^n\text{ with } x\ne y.
$$

We first prove 
\begin{equation}\label{initial-hopf-est}
    |w(x)|\leq Cx_n \quad  \text{in }  B_{1/2}^+. 
\end{equation}
Fix any $q=(q',0)$ with $|q'|\leq1/2$.  Let $R=1/4$, $y :=q-Re_n$, and
$$
 \psi(x):=
 \frac{R^{-b}-|x-y|^{-b}}{R^{-b}-(\sqrt{2}R)^{-b}}, \qquad x\ne y.
$$
Then $\Pplus(D^2\psi)<0$
in $B_R(q)^+\coloneqq B_R(q)\cap\{x_n>0\}$. Clearly, $\psi\ge 0$ on $B_R(q)^+$ and its flat boundary. 
On its spherical boundary, we have $\psi\geq1$, since
$|x-y|^2=2R^2+2Rx_n\geq2R^2$. 
By the maximum principle, we obtain $\pm w\leq\psi$ in $B_R(q)^+$. 
For $x=q+x_ne_n$ with
$0<x_n<R$, by $\psi(q)=0$ and the mean value theorem, we have $\psi(x)\leq C x_n$. Hence, we obtain $|w(x)|\le C x_n$ for $x\in B^+_{1/2}$ with $x_n<R$. 
When $x_n\geq R$, the desired estimate follows from  $|w|\leq 1$ by enlarging $C$ if necessary. Thus \eqref{initial-hopf-est} is proved.

Next suppose $v$ satisfies \eqref{eq:Pucci-class} and $ ax_n\leq v\leq dx_n$ in $B_1^+$, for some slopes $-\infty<a<d<+\infty$.
Set $L:=d-a>0$. We prove that at least one of the following holds in $B_{1/8}^+$:
$$v(x)\ge (a+c L)x_n\quad \text{or}\quad v(x)\le (d-cL)x_n,$$
where $c\in (0,1/2]$ depends only on $n,\lambda,\Lambda$; in particular, independent of $L$. 

Indeed, still set $R=1/4$, then
at $R e_n$, one of $v-ax_n$ and $dx_n-v$
is at least $LR/2$. Both functions are nonnegative solutions of
\eqref{eq:Pucci-class} in $B_1^+$. Let $v_0$ be the larger function
at $R e_n$. Then $v_0$ is a nonnegative solution of \eqref{eq:Pucci-class} in $B_1^+$ and satisfies $v_0(R e_n)\ge LR/2$. 
The weak Harnack inequality and local maximum principle
\cite[Theorem~4.8]{Caffarelli1995FullyNE}, combined along
a finite chain of balls inside $B_{3/4}^+$, give
$v_0\geq c_0 L$ on $B_{R/2}(z)$, for all
 $z=(z', R)$ with $|z'|\leq R/2$. Here $c_0>0$ depends only on $n,\lambda,\Lambda$.
 In the annulus $R/2<|x-z|< R$, contained
in $B_1^+$, the function
$$
 \zeta(x) :=
 \frac{|x-z|^{-b}-R^{-b}}{(R/2)^{-b}-R^{-b}}
$$
satisfies $\Pminus(D^2\zeta)>0$.
Clearly, $\zeta=1$ on the inner sphere, and $\zeta=0$ on the outer sphere.
By the maximum principle, $v_0 \geq c_0 L\zeta$ on the annulus. For $x=(z',x_n)$ with $0<x_n <R/2$,
$\zeta(x)\geq c_1 x_n$. Consequently,
$v_0\geq c Lx_n$ in $B_{R/2}^+$, where
$c=\min\{c_0 c_1,1/2\}>0$. This proves, on $B_{1/8}^+$, either the
lower slope $a$ increases by $cL$ or the upper slope $d$ decreases
by $cL$.

The estimate \eqref{initial-hopf-est} provides initial slopes $a_0=-C$ and $d_0=C$ on $B_{1/2}^+$. Applying a rescaled version of what we just proved inductively, we obtain nested intervals $[a_j,d_j]$ such that $a_j x_n\le w(x)\le d_j x_n$ on
$B_{8^{-j}/2}^+$ and
$d_j-a_j\leq2C(1-c)^j$. Then $a_j\uparrow a$ and $d_j\downarrow a$ for some $|a|\le C$. Choose $\alpha_0\in(0,1)$ such that
$1-c\leq8^{-\alpha_0}$. For
every $j\ge 0$, the slope bound gives
$$
 |w(x)-ax_n|\leq2C8^{-j\alpha_0}x_n
 \leq C|x|^{1+\alpha_0}\quad\text{for every } 2x\in B^+_{8^{-j}} \setminus B^+_{8^{-j-1}}.
$$
The proposition is therefore proved.
\end{proof}

\section{Proof of Theorem \ref{thm:main} and Corollary \ref{cor:uniform}}\label{sec:criterion}

\subsection{Auxiliary lemmas}

We first extend  Proposition \ref{lem:flat-expansion} to the following setting.

\begin{lemma}\label{lem:matching}
Let $n\geq2$, $\alpha\in(0,1)$, and $L\geq0$.
Let $\phi:B_{1/2}\subset\R^{n-1}\to\R$
satisfy $|\phi(x')|\leq|x'|^2$ in $B_{1/2}$.
Suppose that $w$ is a bounded continuous viscosity solution of \eqref{eq:Pucci-class} in $B_1$, and satisfies
$$
 |w(x',\phi(x'))|\leq L|x'|^{1+\alpha} 
 \quad \text{for }|x'|<1/2.
$$
Then there exists a number $a\in\R$ such that
$$
 |a|\leq C\bigl(\norm w_{L^\infty(B_1)}+L\bigr)\quad\text{and}\quad 
 |w(x)-ax_n|\leq C\bigl(\norm w_{L^\infty(B_1)}+L\bigr)|x|^{1+\eta}
 \quad\text{for }x\in B_{c_b},
$$
where $\eta\in(0,\alpha)$, $c_b\in(0,1/4)$, and $C>0$
depend only on $n,\lambda,\Lambda,\alpha$. 
\end{lemma}

\begin{proof}
\emph{Step 1: The flat case.
If $\norm w_{L^\infty(B_1)}\leq1$ and
$w=0$ on $B_{1/2}'$, then there exists $a\in \R$ such that
$$
 |a|\leq C\quad \text{and}\quad 
 |w(x)-ax_n|\leq C|x|^{1+\alpha_0}
 \quad \text{for }x\in B_{1/8},
$$
where positive constants $\alpha_0$ and $C$ depend only on $n,\lambda,\Lambda$.
}

Applying Proposition \ref{lem:flat-expansion}, after rescaling and
reflection, on the two sides of the disk $B'_{1/2}$, we obtain
$a_+,a_-\in \R$ such that 
$$|a_\pm|\le C\quad  \text{and}
\quad 
 |w(x)-a_\pm x_n|\leq C|x|^{1+\alpha_0}
 \quad\text{for }x\in B_{1/8} \text{ with } \pm x_n\geq 0.
$$

We next prove that $a_+=a_-$. If not, we only need to consider the case when $a_+>a_-$, since otherwise we may replace $w$ by $-w$. 

For $r>0$, let
$$w_r(x):={w(rx)}/{r} \text{ for }x\in B_{1/r}, \quad \text{and} \quad  
 w_0 (x):=
a_\pm x_n  \text{ for }\pm x_n\geq 0,
$$
Then the above estimate yields $w_r\to w_0$ in $C^0_{\mathrm{loc}}(\R^n)$ as $r\downarrow 0$.
Since $w_r$ satisfies \eqref{eq:Pucci-class} in
$B_{1/r}$, the viscosity stability gives
 that $w_0$ satisfies \eqref{eq:Pucci-class} in $\R^n$ in the viscosity sense; see \cite{Caffarelli1995FullyNE, CIL}.
 Take
$a_* \in (a_-,a_+)$. Then the quadratic 
$Q(x):=a_*x_n+x_n^2$
touches $w_0$ from below at $0$. The lower Pucci
inequality of $w_0$ gives
$$
 0\geq\Pminus(D^2Q)=2\lambda,
$$
a contradiction.
Therefore, we must have $a_+=a_-=:a$, and the desired conclusion follows.

\medskip

The following proof utilizes ideas in
\cite[Lemma~3.4 and Theorem~3.5]{SS} and their proofs.
Let 
$$\eta:= 2^{-1}\min\{\alpha,\alpha_0\}>0.$$

\emph{Step 2: Approximation at one scale.
There exist $\theta\in(0,1/8)$, $\epsilon\in(0,1/4)$,
and $C_0>0$, depending only on $n,\lambda,\Lambda,\alpha$,
with the following property.
Suppose that $v\in C^0(B_1)$ satisfies
\eqref{eq:Pucci-class} in $B_1$,
$\norm v_{L^\infty(B_1)}\leq1$, and
$\psi:B_{1/2}\subset\R^{n-1}\to\R$ satisfies
$$
 |\psi(x')|\leq\epsilon
 \quad\text{and}\quad
 |v(x',\psi(x'))|\leq\epsilon
 \quad\text{for }|x'|<1/2.
$$
Then there exists $b\in\R$ such that
$$
 |b|\leq C_0
 \quad\text{and}\quad
 \norm{v-bx_n}_{L^\infty(B_\theta)}
 \leq\theta^{1+\eta}.
$$
}

Take $C_0$ to be the constant $C$ in
Step 1, and choose $\theta\in(0,1/8)$
 small enough such that
$$
 C_0\theta^{\alpha_0-\eta}\leq 1/2.
$$
If the assertion failed for every $\epsilon\in(0,1/4)$,
there would be $\epsilon_j\downarrow0$ and functions
$v_j,\psi_j$ satisfying the hypotheses with
$\epsilon=\epsilon_j$, but
$\norm{v_j-bx_n}_{L^\infty(B_\theta)}
 >\theta^{1+\eta}$ for every $|b|\leq C_0$.

The interior H\"older estimate
\cite[Proposition~4.10]{Caffarelli1995FullyNE}
and Arzel\`a--Ascoli give, after passing to a subsequence,
$v_j\to v_\infty$ in $C^0_{\mathrm{loc}}(B_1)$.
The limit function $v_\infty$ satisfies \eqref{eq:Pucci-class} in $B_1$
by viscosity stability, and
$\norm{v_\infty}_{L^\infty(B_1)}\leq1$.
The same interior H\"older estimate gives constants
$\sigma\in(0,1)$ and $C>0$, independent of $j$, such that
$$
 |v_j(x',0)|
 \leq |v_j(x',\psi_j(x'))|
       +C|\psi_j(x')|^\sigma
 \leq\epsilon_j+C\epsilon_j^\sigma
 \quad\text{for }|x'|<1/2.
 $$
Hence, $v_\infty=0$ on $B_{1/2}'$.
By Step 1, there exists $b_\infty\in\R$ with $|b_\infty|\leq C_0$ and
$$
 \norm{v_\infty-b_\infty x_n}_{L^\infty(B_\theta)}
 \leq C_0\theta^{1+\alpha_0}
 \leq 2^{-1}\theta^{1+\eta}.
$$
For large $j$, local uniform convergence therefore gives
$$
 \norm{v_j-b_\infty x_n}_{L^\infty(B_\theta)}
 \leq\norm{v_j-v_\infty}_{L^\infty(B_\theta)}
       +2^{-1}\theta^{1+\eta}
 \leq\theta^{1+\eta},
$$
a contradiction. Step 2 is completed. 

\medskip
\emph{Step 3: Iteration and completion of the proof.}
Let
$$
 A_0:={C_0}/{(1-\theta^\eta)},\quad
 \tau:={\epsilon}/{(1+A_0)},\quad \text{and} \quad 
 N:=\norm w_{L^\infty(B_1)}+L.
$$
Clearly, $0<\tau<1/4$.
If $N=0$, then the desired conclusion holds with $a=0$.
Hence, we may assume $N>0$. Consider 
$$
 v(x):= N^{-1}{w(\tau x)}\quad \text{for } x\in B_1,\quad \text{and}
 \quad
 \psi(x'):=\tau^{-1} {\phi(\tau x')} \quad \text{for }
 |x'|<1/2.
$$
Then $v$ satisfies \eqref{eq:Pucci-class} in $B_1$ and 
$\norm v_{L^\infty(B_1)}\leq 1$. Moreover, 
$$
 |\psi(x')|\leq\tau|x'|^2\quad\text{and}\quad 
 |v(x',\psi(x'))|
 \leq\tau^{1+\alpha}|x'|^{1+\alpha}
 \quad\text{for }|x'|<1/2.
$$

Let $r_j:=\theta^j$ for $j\ge 0$. We construct numbers $a_j$, with
$a_0=0$, such that for every $j\geq0$,
$$
 \norm{v-a_jx_n}_{L^\infty(B_{r_j})}
 \leq r_j^{1+\eta}
 \quad\text{and}\quad
 |a_{j+1}-a_j|\leq C_0r_j^\eta.
$$
We give a construction by induction.
The first estimate holds for $j=0$.
Suppose that $a_0,\dots,a_j$ have been constructed,
then
$$ \textstyle
 |a_j|\leq C_0\sum_{k=0}^{j-1}r_k^\eta\leq A_0.
$$
Define
\[
 v_j(x):= r_j^{-(1+\eta)}
{\bigl(v(r_jx)-a_jr_jx_n \bigr)}\quad \text{and}
 \quad
 \psi_j(x'):= r_j^{-1}\psi(r_jx').
\]
Then $v_j$ satisfies \eqref{eq:Pucci-class} in $B_1$
and $\norm{v_j}_{L^\infty(B_1)}\leq1$.
Moreover, for $|x'|<1/2$,
 \begin{align*}
 |\psi_j(x')|
 &\leq\tau r_j|x'|^2\leq\epsilon,\quad \text{and}\\
 |v_j(x',\psi_j(x'))|
 &\leq\tau^{1+\alpha}r_j^{\alpha-\eta}|x'|^{1+\alpha}
       +A_0\tau r_j^{1-\eta}|x'|^2
 \leq(1+A_0)\tau=\epsilon.
 \end{align*}
In the last inequality, we used $\alpha>\eta$, $1>\eta$, and $\tau<1$.
Step 2 then gives $b_j\in\R$ such that
$$
 |b_j|\leq C_0\quad \text{and}\quad
 \norm{v_j-b_jx_n}_{L^\infty(B_\theta)}
 \leq\theta^{1+\eta}.
$$
Let $a_{j+1}:=a_j+r_j^\eta b_j$. Then
$$
 \norm{v-a_{j+1}x_n}_{L^\infty(B_{r_{j+1}})}
 =r_j^{1+\eta}
   \norm{v_j-b_jx_n}_{L^\infty(B_\theta)}
 \leq r_j^{1+\eta}\theta^{1+\eta}
 =r_{j+1}^{1+\eta},
$$
and $|a_{j+1}-a_j|\leq C_0r_j^\eta$.
This completes the induction, and thus the desired $a_j$'s are constructed.

Since $\{a_j\}$ is Cauchy, we have $a_j\to a_\infty$ for
some $a_\infty\in\R$. Moreover, we have
$$ \textstyle
 |a_\infty|\leq A_0\quad \text{and}\quad 
 |a_\infty-a_j|
 \leq C_0\sum_{k=j}^{\infty}r_k^\eta
 =A_0r_j^\eta\quad \text{for every }j.
$$
For $0<|x|<1$, choose $j\geq0$ such that
$r_{j+1}\leq|x|<r_j$. Then
$$
 |v(x)-a_\infty x_n|
 \leq |v(x)-a_jx_n|+|a_j-a_\infty||x_n|
 \leq r_j^{1+\eta}+A_0r_j^\eta|x|
 \leq C|x|^{1+\eta}.
$$
This estimate also holds at $x=0$, since $v(0)=0$.

Finally, let $a :=Na_\infty/\tau$. Then
$|a|\leq {A_0N}/{\tau}$
and, for $x\in B_\tau$,
$$
 |w(x)-ax_n|
 =N\left|v(x/\tau)-a_\infty{x_n}/\tau\right|
 \leq CN\tau^{-1-\eta}|x|^{1+\eta}.
$$
This proves 
the desired estimates with $c_b:=\tau$.
\end{proof}

The next lemma combines Nirenberg's theorem with Caffarelli's approximation method; see
\cite{Caffarelli-89, Caffarelli1995FullyNE}. In this lemma, we write $x=(z,t)\in\R^2\times\R^{n-2}$.

\begin{lemma}\label{lem:planar}
Let $n\geq3$, $\alpha\in(0,1)$, $D\geq0$, and let
$G:\Sym(n)\to\R$ satisfy \eqref{eq:ellipticity}.
Suppose that $w$ is a bounded continuous viscosity solution
of $G(D^2w)=0$ in $B_1$, and satisfies
\begin{equation}\label{eq:independence}
 |w(z,t)-w(z,0)|\leq D|(z,t)|^{2+\alpha}
 \quad\text{for }(z,t)\in B_1.
\end{equation}
Then there exists a quadratic polynomial $Q=Q(z)$ such that  $G(D^2Q)=0$,
\begin{gather*}
 |Q(0)|+|DQ(0)|+\norm{D^2Q}
 \leq C\bigl(\norm w_{L^\infty(B_1)}+D\bigr),\quad \text{and} \\
 |w(x)-Q(x)|\leq C\bigl(\norm w_{L^\infty(B_1)}+D\bigr)
 |x|^{2+\beta}
 \quad\text{for }x\in B_{1/4},
\end{gather*}
where $\beta\in(0,\alpha)$ and $C>0$ depend only on
$n,\lambda,\Lambda,\alpha$.
Here $Q$ is extended independently of $t$, so that its Hessian
is an $n\times n$ matrix.
\end{lemma}

\begin{proof}
We first record a useful observation: For any continuous viscosity solution $v$ of
$G(D^2v)=0$ in $B_1$, we have
\begin{equation}\label{eq:zero-bound}
 |G(0)|\leq C(n,\Lambda)\norm v_{L^\infty(B_1)}.
\end{equation}
Indeed, if $A:=\norm v_{L^\infty(B_1)}>0$, the maximum of
$v-16A|x|^2$ and the minimum of $v+16A|x|^2$ on
$\overline B_{1/2}$ are attained in the interior.
From this and the equation of $v$, we obtain $G(32AI)\geq0\geq G(-32AI)$. Combining this with a
global Lipschitz bound of $G$ yields \eqref{eq:zero-bound}. The cases when $A=0$ and $A=\infty$ are trivial, so \eqref{eq:zero-bound} is proved.

By Theorem \ref{thm:nirenberg}, there exist $\sigma\in (0,1)$ and
$C_1>0$, depending only on $\lambda,\Lambda$, with the following property.
For every 
$F:\Sym(2)\to\R$ satisfying \eqref{eq:ellipticity},
and every continuous viscosity solution $v$ of
$F(D^2 v)=0$ in $B_1\subset\R^2$ with
$\norm v_{L^\infty(B_1)}\leq 1$, we have $v\in C^{2,\sigma}_{\mathrm{loc}}$ and the quadratic 
$$
 Q_v(z):=v(0)+Dv(0)\cdot z+ 2^{-1} z^T D^2 v(0)z
$$
satisfies $F(D^2 Q_v)=0$,
$$
 \norm{Q_v}_{L^\infty(B_1)}\leq C_1,
 \quad\text{and}\quad
 |v(z)-Q_v (z)|\leq C_1|z|^{2+\sigma}
 \quad\text{for }z\in B_{1/4}.
$$

Let 
$$\beta:=2^{-1}\min\{\alpha,\sigma\}>0. $$

\emph{Step 1: Approximation at one scale.
There exist $\theta\in(0,1/8)$, $\epsilon\in(0,1)$,
and $C_0>0$, depending only on $n,\lambda,\Lambda,\alpha$,
with the following property.
If $v$ is a continuous viscosity solution of $G(D^2v)=0$ in $B_1$, and satisfies
$$
 \norm v_{L^\infty(B_1)}\leq1
 \quad\text{and}\quad
 |v(z,t)-v(z,0)|\leq\epsilon
 \quad\text{for }(z,t)\in B_1,
$$
then there exists a quadratic polynomial $Q=Q(z)$ such that
$G(D^2Q)=0$,
$$
 \norm Q_{L^\infty(B_1)}\leq C_0 \quad \text{and} \quad 
 \norm{v-Q}_{L^\infty(B_\theta)}\leq\theta^{2+\beta}.
$$
}

Take $C_0:=C_1+1$, and choose $\theta\in(0,1/8)$ small enough so that
$C_1\theta^{\sigma-\beta}\leq1/4$.
If the assertion failed for every $\epsilon\in(0,1)$, then
there would be a sequence $(G_j,v_j)$ satisfying all hypotheses with
$\epsilon_j\downarrow 0$, but
$\|v_j-Q\|_{L^\infty(B_\theta)}>\theta^{2+\beta}$ for every quadratic $Q=Q(z)$ satisfying $G_j(D^2 Q)=0$
and $\|Q\|_{L^\infty (B_1)}\le C_0$. 

The interior H\"older estimate
\cite[Proposition~4.10]{Caffarelli1995FullyNE}, a uniform global Lipschitz bound of $G_j$, and \eqref{eq:zero-bound} yield, after passing to a subsequence,
$$
 v_j\to v_\infty\quad\text{in }C^0_{\mathrm{loc}}(B_1),\quad \text{and}
 \quad
 G_j\to G_\infty\quad\text{locally uniformly on }\Sym(n),
$$
where $G_\infty:\Sym(n)\to \R$ satisfies \eqref{eq:ellipticity} with the same constants $\lambda$ and $\Lambda$.
Then the viscosity stability gives
$G_\infty(D^2v_\infty)=0$ in $B_1$.

Since $|v_j(z,t)-v_j(z,0)|\le \epsilon_j$ for $(z,t)\in B_1$, we have
$v_\infty(z,t)=\bar{v}(z)$ with $\norm {\bar{v}}_{L^\infty(B_1)}\leq1$.
Let $\overline G(M):=G_\infty(\diag\{M,0\})$ for $M\in \Sym(2)$, then it satisfies \eqref{eq:ellipticity} with constants
$\lambda,\Lambda$, and 
$$
 \overline G(D^2 \bar{v})=0\quad\text{in }B_1\subset\R^2.
$$
Then, as discussed at the beginning of the proof,
Nirenberg's theorem yields $\bar{v}\in C^{2,\sigma}_{\mathrm{loc}}$ and the quadratic $Q_{\bar{v}}$ satisfies
$\overline{G}(D^2Q_{\bar v})=0$, 
$$
 \norm{Q_{\bar v}}_{L^\infty(B_1)}\leq C_1,\quad\text{and} \quad 
 |\bar{v}(z)-Q_{\bar{v}}(z)|\leq C_1|z|^{2+\sigma}
 \quad\text{for }|z|<1/4.
$$

Let $J:=\diag\{I_2,0\}$ and $Q_\infty(z,t)\coloneqq Q_{\bar{v}}(z)$. For each $j$, the uniform ellipticity of $G_j$ gives a unique $s_j\in\R$ such that
$$
 G_j(D^2Q_\infty+s_jJ)=0\quad\text{and}\quad 
2\lambda |s_j|\leq {|G_j(D^2Q_\infty)|}\to  0.
$$
Therefore,
the quadratic $Q_j:=Q_\infty+s_j|z|^2/2$ satisfies
$G_j(D^2Q_j)=0$ and $\norm{Q_j}_{L^\infty(B_1)}\leq C_0$
for large $j$. However, by the choice of $\theta$,
$$
 \norm{v_j-Q_j}_{L^\infty(B_\theta)}
 \leq\norm{v_j-v_\infty}_{L^\infty(B_\theta)}
      + 4^{-1}\theta^{2+\beta}
      + 2^{-1}|s_j|\theta^2
 \leq\theta^{2+\beta}
$$
for large $j$, a contradiction. Step 1 is completed.

\medskip
\emph{Step 2: Iteration and completion of the proof.}

Let $N:=\norm w_{L^\infty(B_1)}+D/\epsilon$.
If $N=0$, then $w\equiv 0$, $G(0)=0$, and thus the desired conclusion of the lemma holds with $Q=0$.
Hence, we may assume $N>0$.
Replacing $w$ by $w/N$ and $G(M)$ by $G(NM)/N$,
we may further assume that
$$\norm w_{L^\infty(B_1)}\leq1 \quad \text{and} \quad D\leq\epsilon.$$

Let $r_j:=\theta^j$ for $j\geq0$.
By induction, we prove that there exist quadratic polynomials $Q_j=Q_j(z)$, with
$Q_0=0$, such that $G(D^2Q_j)=0$ for $j\geq 1$, and
$$
 \norm{w-Q_j}_{L^\infty(B_{r_j})}\leq r_j^{2+\beta}
 \quad\text{for }j\geq0,
$$
The assertion holds for $j=0$.
Suppose that $Q_j$ has been constructed. Let
$$
 w_j(x):=r_j^{-2-\beta}\bigl(w(r_jx)-Q_j(r_jz)\bigr)\quad \text{and}
 \quad
 G_j(M):=r_j^{-\beta}G(D^2Q_j+r_j^\beta M).
$$
Then $G_j(D^2w_j)=0$ in $B_1$,
$\norm{w_j}_{L^\infty(B_1)}\leq1$, and $G_j$ has the
same ellipticity constants as those of $G$.
Since $Q_j$ is independent of $t$, assumption \eqref{eq:independence} implies
$$
 |w_j(z,t)-w_j(z,0)|
 \leq\epsilon r_j^{\alpha-\beta}|(z,t)|^{2+\alpha}
 \leq\epsilon\quad\text{for }(z,t)\in B_1.
$$
Step 1 then gives a quadratic $q_j=q_j(z)$ satisfying
$G_j(D^2q_j)=0$,
$\norm{q_j}_{L^\infty(B_1)}\leq C_0$, 
and
$\norm{w_j-q_j}_{L^\infty(B_\theta)}
 \leq\theta^{2+\beta}$.
Define
$$
 Q_{j+1}(z):=Q_j(z)+r_j^{2+\beta}q_j(z/r_j).
$$
Then
$G(D^2Q_{j+1})=0$ and
$$
 \norm{w-Q_{j+1}}_{L^\infty(B_{r_{j+1}})}
 =r_j^{2+\beta}\norm{w_j-q_j}_{L^\infty(B_\theta)}
 \leq r_{j+1}^{2+\beta}.
$$
This completes the induction.

For every $j\geq 0$, the triangle inequality gives
$\norm{Q_{j+1}-Q_j}_{L^\infty(B_{r_{j+1}})}
 \leq 2r_j^{2+\beta}$. 
Then a standard fact of quadratic polynomials yields
$$
 |(Q_{j+1}-Q_j)(0)|+
 r_j |D(Q_{j+1}-Q_j)(0)| +
 r_j^2 \norm{D^2(Q_{j+1}-Q_j)} \leq C_* r_j^{2+\beta},
$$
where $C_*$ is independent
of $j$.
Therefore, the coefficients of $Q_j$ converge to those of a
quadratic polynomial $Q=Q(z)$.
Since $Q_0=0$, summing the above coefficient bounds gives
$$
|Q(0)|+|DQ(0)|+\norm{D^2Q}\leq C.
$$
Since $G(D^2Q_j)=0$ for all $j\geq 1$, continuity of $G$
gives $G(D^2Q)=0$.

For $x=(z,t)$ with $|x|\leq r_j$, the coefficient bounds
imply
$$ \textstyle
 |Q(z)-Q_j(z)|
 \leq C \sum_{k=j}^{\infty}
 \bigl(r_k^{2+\beta}
       +r_k^{1+\beta}|x|
       +r_k^\beta|x|^2\bigr)
 \leq C r_j^{2+\beta}.
$$
For $0<|x|<1/4$, choose $j\geq0$ such that
$r_{j+1}\leq|x|<r_j$. Then
$$
 |w(x)-Q(x)|
 \leq |w(x)-Q_j(x)|+|Q_j(x)-Q(x)|
 \leq Cr_j^{2+\beta}
 \leq C|x|^{2+\beta}.
$$
This estimate also holds at $0$ by continuity.
The lemma is therefore proved.
\end{proof}

The next lemma translates a bound of the second fundamental form near $p$ to a quantitative local graphical representation near $p$. 

\begin{lemma}\label{lem:curvature-graphs}
Let $S\subset\R^n$ be a $C^2$ hypersurface with admissible radius $\rho>0$ at $p\in S$.
After a rigid motion taking $p$ to the origin and $T_pS$ to
$\{x_n=0\}$, the set $S$ contains the graph of a $C^2$ function
$\phi: B_{\rho/8}\subset\R^{n-1}\to \R$ satisfying $\phi(0)=0$, $D\phi(0)=0$, and
$\rho \norm{D^2\phi}\leq 3$ in $B_{\rho/8}$.
\end{lemma}

Here and throughout the paper, for a linear map $A:\R^m\to\R^k$, we use
its operator norm $\norm A:=\sup_{|\xi|=1}|A\xi|$.
Recall that for a $C^2$ graph $z=(x',\phi(x'))$ and the tangent
vector $\widehat\xi=(\xi,D\phi(x')\cdot\xi)$, the graph
formula for the second fundamental form is, up to a sign,
$$
 \II_S(z)(\widehat\xi,\widehat\xi)
 =\frac{D^2\phi(x')(\xi,\xi)}{\sqrt{1+|D\phi(x')|^2}},\quad \xi\in \R^{n-1}.
$$

\begin{proof}[Proof of Lemma \ref{lem:curvature-graphs}]
By translation and rotation, we may assume
$p=0$ and $T_0S=\{x_n=0\}$. Note that $\rho \sup_{S\cap B_\rho(p)}\norm{\II_S}$ is scale-invariant, we may assume $\rho=1$ by a rescaling.

\emph{Claim: The set $S\cap B_{1/4}$ contains a $C^2$ graph
$(x',\phi(x'))$ for $|x'|<1/8$, where $\phi$ satisfies
$\phi(0)=0$, $D\phi(0)=0$, and $|D\phi|\leq1/3$ in $B_{1/8}\subset \R^{n-1}$.
}

Assuming this claim, by the graph formula of the second fundamental form and $\norm{\II_S}\leq 1$, we obtain, for every $\xi\in \R^{n-1}$ and $|x'|<1/8$,
$$
 |D^2\phi(x')(\xi,\xi)|
 \leq\sqrt{1+|D\phi(x')|^2} \, |\widehat\xi|^2
 \leq(1+|D\phi(x')|^2)^{3/2}|\xi|^2
 \leq 3|\xi|^2.
$$
Hence $\norm{D^2\phi}\leq 3$ in $B_{1/8}$.
Rescaling $\rho$ back proves the lemma.

Now we prove the claim.
For $q\in S\cap B_1$, let $\Pi_q^\perp$ be the orthogonal
projection onto the normal line $(T_qS)^\perp$.
Since $S$ is $C^2$, this projection depends $C^1$ on $q$.
Let
$$
 \mathcal U
 :=\{q\in S\cap B_1:|\Pi_q^\perp e_n|>1/2\},
 \quad \text{and}\quad 
 N(q):={\Pi_q^\perp e_n}/{|\Pi_q^\perp e_n|}
 \quad\text{for }q\in\mathcal U,
$$
where $e_n\coloneqq (0,\dots,0,1)$. 
Then $\mathcal U$ is open in $S$ and $0\in \mathcal U$.
The vector field $N$ is a $C^1$ unit normal on $\mathcal U$,
and satisfies 
$N(0)=e_n$ and
$N(q)\cdot e_n
 =|\Pi_q^\perp e_n|>1/2$ for $q\in \mathcal U$.
Below, write $N=(N',N_n)$ with $N'\in\R^{n-1}$.

Let $\pi:\R^n\to\R^{n-1}$ be the horizontal projection defined by
$\pi(x',x_n)=x'$.
For each $a\in\R^{n-1}$ with $|a|<1/8$, define on
$\mathcal U$ the vector field
$$
 V_a(q):=\left(a,- \bigl( {N'(q)\cdot a} \bigr)/ {N_n(q)} \right),\quad q\in \mathcal U.
$$
Clearly, it is $C^1$ and tangent to $S$ on $\mathcal U$.
The local existence and uniqueness theorem for ODE
therefore gives a unique maximal solution  $q_a:[0,T_a)\to\mathcal U$ of
$$
 q_a'(t)=V_a(q_a(t))\quad \text{and}\quad  q_a(0)=0.
$$
Its horizontal component satisfies
$d\pi(q_a(t))/dt=a$, so $\pi(q_a(t))=ta$.

For $q\in\mathcal U$, the identity $|N(q)|=1$ gives
$|V_a(q)|^2
 \leq {|a|^2}/{N_n(q)^2}
 \leq 4|a|^2$. From $\norm{\II_S}\le 1$, 
we have
$|dN_q(v)|\leq\norm{\II_S(q)} |v|\leq|v|$
for every $v\in T_qS$.
Therefore, for every $t\in [0,T_a)$,
$$
 |q_a'(t)|\leq2|a|\quad \text{and}
 \quad
 \left| dN(q_a(t))/dt\right|
 \leq2|a|.
$$
Integrating from zero, we obtain, for
$0\leq t<\min\{T_a,1\}$,
$$
 |q_a(t)|\leq2t|a|<1/4,\quad
 |N(q_a(t))-e_n|\leq2t|a|<1/4,\quad\text{and}\quad 
 N_n(q_a(t))>3/4.
$$

We next prove $T_a>1$. If not, then the above bounds on $q_a$ and $q_a'$
imply that $q_a(t)$ has a limit
$q_*\in\overline B_{1/4}$ as $t\uparrow T_a$.
Since $S\cap B_1$ is relatively closed in $B_1$,
we have $q_*\in S\cap B_1$.
By continuity of the normal projections,
$$
 |\Pi_{q_*}^\perp e_n|
 =\lim_{t\uparrow T_a}|\Pi_{q_a(t)}^\perp e_n|
 =\lim_{t\uparrow T_a}N_n(q_a(t))
 \geq3/4.
$$
Thus $q_*\in\mathcal U$.
Then the solution $q_a(t)$ extends past $T_a$,
contradicting its maximality.

We have proved $T_a>1$.
In particular, $q_a(1)$ exists for every $|a|<1/8$, and satisfies
$$
\pi(q_a(1))=a,\quad
|q_a(1)|<1/4,\quad\text{and}\quad 
|N(q_a(1))-e_n|<1/4.
$$

Define $\phi$ on $B_{1/8}\subset\R^{n-1}$ by
$$
 q_a(1)=(a,\phi(a)),\quad |a|<1/8.
$$
Then the graph of $\phi$ is contained in $S\cap B_{1/4}$.

We next check that $\phi$ is $C^2$.
By continuous dependence of solutions of ODEs on parameters, the map
$a\mapsto q_a(1)$ is continuous.
Fix $a_0\in B_{1/8}$ and set $q_0=q_{a_0}(1)$.
Since $S$ is a $C^2$ hypersurface, there are an open neighborhood
$U$ of $q_0$ and a function $f\in C^2(U)$ such that
$S\cap U=\{x\in U:f(x)=0\}$ and
$\nabla f(q_0)\neq 0$.
The vector $\nabla f(q_0)$ is normal to $S$ at $q_0$ and hence
is a nonzero multiple of $N(q_0)$.
Since $N_n(q_0)>3/4$, we have $\partial_n f(q_0)\neq0$.
The implicit function theorem therefore gives an open
neighborhood $V$ of $a_0$, an open interval $I$ containing
$\phi(a_0)$, and a function $g\in C^2(V)$ such that
$S\cap(V\times I)=\{(a,g(a)):a\in V\}$.
By continuity of $a\mapsto q_a(1)$, for all $a$ sufficiently
close to $a_0$ we have
$(a,\phi(a))=q_a(1)\in S\cap(V\times I)$.
Thus $\phi(a)=g(a)$ near $a_0$.
Since $a_0$ was arbitrary, $\phi\in C^2(B_{1/8})$.

Finally, since $N(q_a(1))$ is normal to the graph, we have
$D\phi(a)=- {N'(q_a(1))}/{N_n(q_a(1))}$.
Since $N_n(q_a(1))>3/4$ and $|N'(q_a(1))|<1/4$, we obtain
$$
 |D\phi(a)|
 \leq 1/3,\quad a\in B_{1/8}\subset \R^{n-1}.
$$
For $a=0$, we have $q_0(t)\equiv 0$, and thus $\phi(0)=0$. By $N(0)=e_n$, we have $D\phi(0)=0$.
The claim is proved.
\end{proof}

\subsection{Proof of Theorem \ref{thm:main}}\label{sec:main-proof}

We first define the derivative and seminorm in
\eqref{eq:field-bound}.
Let $S$ be a $C^2$ hypersurface in $\R^n$, and let  $\nu\in C^1(S;\R^n)$. For $x\in S$, let
$\Pi_x:\R^n\to T_xS$ be the Euclidean orthogonal projection.
For $\xi\in\R^n$, define
$$
 D_S\nu(x)\xi
 :=\left.\frac{d}{dt}\nu(\sigma(t))\right|_{t=0},
$$
where $\sigma$ is any $C^1$ curve in $S$ satisfying  $\sigma(0)=x$ and $\sigma'(0)=\Pi_x\xi$. One can show that $D_S\nu(x)$ is
independent of $\sigma$ and is linear in $\xi$.
Thus $D_S\nu(x)$ is a linear map from $\R^n$ to $\R^n$.

For $E\subset S$ and $0<\gamma\leq1$, define 
$$
 [D_S\nu]_{\gamma;E}
 :=\sup_{\substack{x,y\in E\\x\ne y}}
 \frac{\norm{D_S\nu(x)-D_S\nu(y)}}{|x-y|^\gamma}.
$$

\begin{proof}[Proof of Theorem~\ref{thm:main}]
We may assume $0<\norm u_{L^\infty(B_1)}<\infty$.
Replacing $u$ by $u/\norm{u}_{L^\infty(B_1)}$ and replacing $F(M)$ by
$F(\norm u_{L^\infty(B_1)}M)/\norm u_{L^\infty(B_1)}$,
we may further assume $\norm u_{L^\infty(B_1)}=1$.
For $1\le i\le n-2$, write
$$\mu_i:=\nu_i(0).
$$

\emph{Step 1.
For each $i$, there exists a constant vector $g_i\in\R^n$ with $|g_i|\leq C(1+K)\rho^{-1}$ such that
$$
| \partial_{\mu_i}u(x)
 -\partial_{\mu_i}u(0)- g_i\cdot x |
 \leq C(1+K)\rho^{-1-\eta}|x|^{1+\eta}
 \quad\text{for }x\in B_{c_b\rho/8},
$$
where $\eta\in(0,\gamma)$, $c_b\in(0,1/4)$, and $C>0$
depend only on $n,\lambda,\Lambda,\gamma$.
}

Fix any $1\le i\le n-2$.
Let $$r_0:=\rho/8, \quad a:=Du(0), \quad  \text{and} \quad A_i:=D_{S_i}\nu_i(0).$$
By Lemma~\ref{lem:curvature-graphs}, a graphical representation of $S_i$
near zero, in orthonormal coordinates over $T_0S_i$,
takes the form $(x',\phi_i(x'))$ for $|x'|<r_0$ satisfying
\begin{equation}\label{graph-bounds}
 \phi_i(0)=0,\quad D\phi_i(0)=0,\quad \text{and}\quad 
 \norm{D^2\phi_i(x')}\leq 3/\rho\quad \text{for }|x'|<r_0.
\end{equation}

For a point $x=(x',\phi_i(x'))$ on this graph, consider
$\Gamma(t):=(tx',\phi_i(tx'))$, $0\leq t\leq1$.
Using \eqref{graph-bounds}, we obtain
$|\Gamma(t)|\leq2t|x|$ and $|\Gamma'(t)|\leq2|x|$ for $0\le t\le 1$.
Since $\Gamma'(t)$ is tangent to $S_i$, the definition
of $D_{S_i}\nu_i$ gives
$$
 \nu_i(x)-\mu_i-A_ix
 =\int_0^1
 \bigl(D_{S_i}\nu_i(\Gamma(t))-A_i\bigr)\Gamma'(t)\,dt.
$$
Using \eqref{eq:field-bound} and the bounds for $\Gamma$ and its derivative, we obtain for every $x$ on this graph that
\begin{equation}\label{graph-est}
 |\nu_i(x)-\mu_i-A_ix|
 \leq CK\rho^{-1-\gamma}|x|^{1+\gamma}
 \quad \text{and}\quad 
 |\nu_i(x)-\mu_i|\leq2K\rho^{-1}|x|.
\end{equation}
These estimates hold in the original
coordinates as well.

Define
$$
 w_i(x):=\mu_i\cdot(Du(x)-a)+a\cdot A_ix
 \quad\text{for }x\in B_{1}.
$$

Note that $\mu_i$ is unit. For $h>0$, the difference quotient $\bigl(u(x+h\mu_i)-u(x)\bigr)/{h}$ satisfies
\eqref{eq:Pucci-class} 
 in $B_{1-h}$ by \cite[Proposition 5.5]{Caffarelli1995FullyNE} (Note that our definition of uniform ellipticity uses $\tr N$, not $\|N\|$).
Since $u\in C^1(B_1)$, this difference quotient converges
locally uniformly to $\partial_{\mu_i} u$ in $B_1$. Passing $h$ to zero and using the 
viscosity stability, 
 $w_i$ satisfies the Pucci inequality
\eqref{eq:Pucci-class} in $B_{1}$.

By the interior $C^{1,\gamma_0}$ estimate \cite[Corollary~5.7]{Caffarelli1995FullyNE}, we have
\begin{equation}\label{u-C-1-gamma-est}
 \norm{Du}_{C^{0,\gamma_0}(B_{1/2})}\leq C,
\end{equation}
where $\gamma_0\in(0,1)$ and $C>0$ depend only on
$n,\lambda,\Lambda$.
By \eqref{eq:field-bound}, we have $\norm{A_i}\leq K/\rho$.
From this and \eqref{u-C-1-gamma-est}, 
we obtain
$$\norm{w_i}_{L^\infty(B_{r_0})}\leq C(1+K).$$

For a point $x$ on the graph in $S_i$, the identity
$\nu_i(x)\cdot Du(x)=\mu_i\cdot a$ gives
$$
 w_i(x)=-(\nu_i(x)-\mu_i)\cdot(Du(x)-a)
        -a\cdot\bigl(\nu_i(x)-\mu_i-A_ix\bigr).
$$
Let $\alpha:=\min\{\gamma,\gamma_0\}\in(0,1)$.
Using \eqref{graph-est} and \eqref{u-C-1-gamma-est}, we obtain
$$
 |w_i(x)|
 \leq CK\bigl(\rho^{-1}|x|^{1+\gamma_0}
              +\rho^{-1-\gamma}|x|^{1+\gamma}\bigr)
 \leq CK\rho^{-1-\alpha}|x|^{1+\alpha}.
$$
Here we used $|x|<\rho\leq1$.

Rescaling by $r_0$, the graph bounds \eqref{graph-bounds} give
$$
 |r_0^{-1}\phi_i(r_0x')|
 \leq\frac{3}{16}|x'|^2
 \quad\text{for }|x'|<1/2,
$$
in graph coordinates.
The rescaled $w_i$ restricted to the rescaled graph is bounded by
$CK|x'|^{1+\alpha}$ for $|x'|<1/2$.
Lemma \ref{lem:matching} therefore gives a number $b_i$
with $|b_i|\leq C(1+K)$ such that, in the original coordinates,
$$
 |w_i(x)-(b_i/r_0)n_i\cdot x|
 \leq C(1+K)r_0^{-1-\eta}|x|^{1+\eta}
 \quad\text{for }x\in B_{c_br_0},
$$
where $n_i$ is the unit normal used for the graph coordinate at zero.
Let $g_i:=(b_i/r_0)n_i-A_i^Ta$. Then the desired conclusion of Step 1 follows.

\medskip

For the rest of the proof, let $r_*:=c_b\rho/8$.
The constant $C$ may now depend additionally on an upper
bound for $K$ and positive lower bounds for $\rho$ and $\delta$.

\medskip
\emph{Step 2.
There exists a quadratic polynomial $Q$ such that
$F(D^2Q)=0$,
$$
 |Q(0)|+|DQ(0)|+\norm{D^2Q}\leq C,\quad\text{and}\quad 
 |u(x)-Q(x)|\leq C|x|^{2+\beta}
 \quad\text{for }x\in B_{r_*/4},
$$
where $\beta\in(0,\gamma)$ depends only on
$n,\lambda,\Lambda,\gamma$.
}

Choose $O\in O(n)$ whose last $n-2$ columns
form an orthonormal basis of
$\operatorname{span}\{\mu_1,\dots,\mu_{n-2}\}$.
By the definition of $\delta$, each of these columns is
$\sum_i\zeta_i\mu_i$ for some $\zeta\in \R^{n-2}$ with $|\zeta|\leq\delta^{-1}$.
Let $v(z,t):=u(O(z,t))$, where
$(z,t)\in\R^2\times\R^{n-2}$.
By Step 1, the chain rule yields
\begin{equation}\label{eq:t-gradient}
 D_tv(z,t)=b+Bz+Tt+R(z,t)\quad \text{with}\quad 
 |R(z,t)|\leq C|(z,t)|^{1+\eta}
 \quad\text{for }(z,t)\in B_{r_*},
\end{equation}
where $|b|+\norm B+\norm T\leq C$.
Let
$$
 Q_1(z,t):=b\cdot t+t^TBz+ 2^{-1} t^TTt.
$$
Integrating \eqref{eq:t-gradient} along $(z,st)$ for $0\le s\le 1$ gives
$
 v(z,t)-v(z,0)-Q_1(z,t)
 =\int_0^1 R(z,st)\cdot t\,ds$, and therefore
$$
 |v(z,t)-v(z,0)-Q_1(z,t)|
 \leq C|(z,t)|^{2+\eta} \quad\text{for }(z,t)\in B_{r_*}.
$$

Consider $w:=v-Q_1$ and  $H(M):=F\bigl(O(M+D^2Q_1)O^T\bigr)$ for $M\in \Sym(n)$. Then 
$$H(D^2w)=0\quad \text{and}\quad  |w(z,t)-w(z,0)|\leq C|(z,t)|^{2+\eta}
 \quad\text{for }(z,t)\in B_{r_*}.$$
The operator $H$ has ellipticity constants $\lambda,\Lambda$.
Applying Lemma \ref{lem:planar} to $w(r_*\cdot)$ and
$r_*^2H(r_*^{-2}\cdot)$ and then rescaling back, we obtain a two-variable quadratic polynomial $Q_2=Q_2(z)$
such that
$H(D^2Q_2)=0$ and
$$
 |v(z,t)-Q_1(z,t)-Q_2(z)|
 \leq C|(z,t)|^{2+\beta}
 \quad\text{for }(z,t)\in B_{r_*/4}.
$$
Here the factors involving $r_*^{-1}$ are absorbed into $C$, since $C$ may depend
on $\rho$.
Taking $Q=(Q_1+Q_2)\circ O^T$ proves Step 2.

\medskip
\emph{Step 3: Completion of the proof.}
By Step 2 and $u\in C^1$, we have 
$Q(0)=u(0)$ and $D Q(0)=D u(0)$.
To identify the Hessian, let $h:=u-Q$.
It satisfies $\widetilde{F}(D^2 h)=0$ in $B_1$ with operator
$\widetilde{F}(M):= F(M+D^2Q)$.
For $0<s<r_*/8$, a rescaled version of \cite[Corollary~5.7]{Caffarelli1995FullyNE} and Step 2 imply
$$
 \norm{Dh}_{L^\infty(B_s)}
 \leq Cs^{-1}\norm h_{L^\infty(B_{2s})}
 \leq Cs^{1+\beta}.
$$
Taking $s=2|x|$ yields
$$
 |Du(x)-Du(0)-D^2Qx|
 \leq C|x|^{1+\beta}
 \quad\text{for }0<|x|<r_*/16.
$$
Hence, $Du$ is differentiable at zero and $D^2u(0)=D^2Q$.
Step 2 now gives \eqref{eq:main-coefficients} and
\eqref{eq:main-expansion}, with $c:=c_b/32$.
Theorem \ref{thm:main} is proved.
\end{proof}

\subsection{Proof of Corollary \ref{cor:uniform}}
For every $p\in B_{1/2}$, applying a rescaled version of
Theorem \ref{thm:main} in $B_{1/4}(p)$ implies  that $Du$ is
differentiable at every $p$,
$$
 |Du(p)|+\norm{D^2u(p)}\leq C\norm u_{L^\infty(B_{3/4})} ,\quad\text{and}
 $$
 $$
 |u(x)-Q_p(x)|\leq C\norm u_{L^\infty(B_{3/4})} |x-p|^{2+\beta}
 \quad\text{for }x\in B_{c\rho/4}(p),
$$
where
$Q_p(x):=u(p)+Du(p)\cdot(x-p)+2^{-1}(x-p)^TD^2u(p)(x-p)$. Here, all constants are as stated in the corollary, and in particular, independent of $p$.
Then the desired conclusion follows by a standard argument.

\section{Proofs of Corollary \ref{cor:quadratic} and \ref{cor:dirichlet}}\label{sec:applications}

\subsection{Proof of Corollary~\ref{cor:quadratic}}
By Corollary~\ref{cor:uniform}, it suffices to construct the following
for every $p\in B_{1/2}$ and $1\le j\le n-2$: A $C^2$ hypersurface $S_j$
through $p$ with admissible radius
$\rho:= 1/{(8L^2)}$,
and a vector field
$\nu_j\in C^{1,\gamma}(S_j\cap B_\rho(p);\R^n)$ satisfying
\eqref{eq:field-bound} with
$K:=8L^6$
such that
$|\nu_j(p)|=1$, $\nu_j(x)\cdot Du(x)=0$ for $x\in S_j\cap B_\rho(p)$,
and
$\delta(\nu_1(p),\dots,\nu_{n-2}(p))\geq\delta:=L^{-2}$.

Fix $p\in B_{1/2}$. Let
$$
A(x):=(D\Phi(x))^{-1}\quad\text{for } x\in B_{3/4}.
$$

\emph{Step 1: Construction of $S_j$ and $\nu_j$.}
For each $i$, choose an orthonormal basis of
$$
\{\eta\in\R^{r_i}:DQ_i(X^{(i)}(p))\cdot\eta=0\}.
$$
This space has dimension at least $r_i-1$.
Regard these vectors as vectors in $\R^n$ by placing them
in the coordinates corresponding to $X^{(i)}$ and setting
all other components equal to zero.
Together with the $m$ standard coordinate vectors
corresponding to $Y$, they form a set
of at least
$$ \textstyle
m+\sum_{i=1}^k(r_i-1)=n-d-k\geq n-2
$$
orthonormal vectors.
Choose $n-2$ of them and denote them by
$\xi_1,\dots,\xi_{n-2}$.

We associate an affine hyperplane $L_j$ through $\Phi(p)$
to each $\xi_j$.
Suppose first that $\xi_j$ has its nonzero component
$\eta\in\R^{r_i}$ in the coordinates corresponding to
$X^{(i)}$.
The affine function
$q\mapsto DQ_i(q^{(i)})\cdot\eta$ vanishes at $\Phi(p)$.
If it is not identically zero, set
$$
L_j:=\{q\in\R^n:DQ_i(q^{(i)})\cdot\eta=0\}.
$$
This is an affine hyperplane through $\Phi(p)$.
If that affine function is identically zero, or if $\xi_j$
is a coordinate vector corresponding to $Y$, take
$$
L_j:=\{q\in\R^n:\xi_j\cdot(q-\Phi(p))=0\}.
$$

We observe that, for every $j$ and $s\in \R$, the points $q+s\xi_j$ and $q-s\xi_j$
give the same arguments in $U$ whenever $q\in L_j$.
Indeed, for $L_j$ arising from $Q_i$, the quadratic identity
$$
Q_i(q^{(i)}+s\eta)-Q_i(q^{(i)}-s\eta)
=2s\,DQ_i(q^{(i)})\cdot\eta=0
$$
holds for every $q\in L_j$ and $s\in\R$.
All other arguments of $U$ are unchanged by adding or
subtracting $s\xi_j$.
For $L_j$ corresponding to $Y$, every argument of
$U$ is clearly unchanged.

For every $1\le j\le n-2$, we may rewrite
$$
L_j=\{q\in\R^n:\ell_j\cdot(q-\Phi(p))=0\}
$$
for some $|\ell_j|=1$. Define
$$
S_j:=\{x\in B_{3/4}:f_j(x)=0\},
$$
where $f_j(x):=\ell_j\cdot(\Phi(x)-\Phi(p))$.
Since $Df_j=D\Phi^T\ell_j\neq0$, the implicit function
theorem shows that $S_j$ is a $C^{2,\gamma}$ hypersurface
through $p$.
Define
$$
\nu_j(x):=d_j^{-1}{A(x)\xi_j}\quad
\text{for }
x\in B_{3/4},
$$
where $d_j:=|A(p)\xi_j|$.
In particular, $|\nu_j(p)|=1$.

Fix $x\in S_j$ and choose a local inverse $\Psi$ of $\Phi$
near $\Phi(x)$, with $\Psi(\Phi(x))=x$.
For sufficiently small $|s|$, our previous observation gives
$$
u\bigl(\Psi(\Phi(x)+s\xi_j)\bigr)
=
u\bigl(\Psi(\Phi(x)-s\xi_j)\bigr).
$$
Since $D\Psi(\Phi(x))=A(x)$, differentiation at $s=0$
gives
$2Du(x)\cdot A(x)\xi_j=0$.
Consequently, $\nu_j\cdot Du=0$ on $S_j$.

\medskip
\emph{Step 2: Verification of the uniform bounds.}
The assumption on $\Phi$ gives
$$
|Df_j|=|D\Phi^T\ell_j|\geq L^{-1}\quad \text{and}
\quad
\norm{D^2f_j}\leq L \quad \text{on }B_{3/4}.
$$
For $x\in S_j$ and $v\in T_xS_j$, the level set formula
for the second fundamental form yields
$$
|\II_{S_j}(x)(v,v)|
={|D^2f_j(x)(v,v)|}/{|Df_j(x)|}
\leq L^2|v|^2.
$$
Also, $S_j$ is relatively closed in $B_{3/4}$.
Since $B_\rho(p)\subset B_{3/4}$ and $\rho L^2=1/8$,
the radius $\rho$ is admissible at $p$ for every $S_j$.

The bounds on $D\Phi$ and $A$ imply
$L^{-1}\leq d_j\leq L$.
For every $a\in\R^{n-2}$, orthonormality of the $\xi_j$
therefore gives
$$ \textstyle
|\sum_j a_j\nu_j(p)|
=|A(p)\sum_j\frac{a_j}{d_j}\xi_j|
\geq L^{-1}
 (\sum_j\frac{|a_j|^2}{d_j^2})^{1/2}
\geq L^{-2}|a|.$$
Thus the independence number of $\nu_1(p),\dots, \nu_{n-2}(p)$ is at least $\delta=L^{-2}$.

We finally verify \eqref{eq:field-bound}.
For every $x,y \in B_{3/4}$, we have
$$
A(x)-A(y)=A(x)\bigl(D\Phi(y)-D\Phi(x)\bigr)A(y),
$$
and differentiation of $D\Phi(x)A(x)=I$ gives
$$
DA(x)[h]
=-A(x)D^2\Phi(x)[h,\cdot]A(x).
$$
Hence, for $x,y\in B_{3/4}$,
$\norm{A(x)-A(y)}\leq L^3|x-y|$ and $\norm{DA(x)}\leq L^3$.
From these, we obtain
$$
\norm{DA(x)-DA(y)}
\leq2L^5|x-y|+L^3|x-y|^\gamma,\quad x,y\in B_{3/4}.
$$
Since $\nu_j=A\xi_j/d_j$ and $d_j\geq L^{-1}$, it follows that
$$\norm{D\nu_j(x)}\leq L^4 \quad \text{and}\quad
\norm{D\nu_j(x)-D\nu_j(y)}
\leq 2L^6|x-y|+L^4|x-y|^\gamma, \quad x,y\in B_{3/4}.$$
In particular, $\nu_j\in C^{1,\gamma}(B_{3/4};\R^n)$.

On $B_{3/4}$, let
$$
\Pi_j:=I- {|Df_j|^{-2}} {Df_j\otimes Df_j}.
$$
For $x\in S_j$, this is the orthogonal projection onto $T_x S_j$.
Differentiating this gives
$
\norm{D\Pi_j}
\leq4 {|Df_j|}^{-1} {\norm{D^2f_j}} 
\leq4L^2$,
and therefore
$$
\norm{\Pi_j(x)-\Pi_j(y)}\leq4L^2|x-y|,\quad x,y\in B_{3/4}.
$$
By the definition of $D_S \nu$, we have
$$
D_{S_j}\nu_j=D\nu_j\,\Pi_j\quad\text{on }S_j.
$$
Thus $\norm{D_{S_j}\nu_j}\leq L^4$ on $S_j$, and for
$x,y\in S_j\cap B_\rho(p)$,
$$
\begin{aligned}
\norm{D_{S_j}\nu_j(x)-D_{S_j}\nu_j(y)}
&\leq \norm{D\nu_j(x)-D\nu_j(y)}
 +\norm{D\nu_j(y)}\,\norm{\Pi_j(x)-\Pi_j(y)}\\
&\leq6L^6|x-y|+L^4|x-y|^\gamma \leq7L^6|x-y|^\gamma.
\end{aligned}
$$
The last inequality uses $|x-y|<2\rho\leq1$ and $L\geq1$.
Consequently,
$$
\rho\norm{D_{S_j}\nu_j}_{L^\infty(S_j\cap B_\rho(p))}
+\rho^{1+\gamma}
 [D_{S_j}\nu_j]_{\gamma;S_j\cap B_\rho(p)} \leq \rho L^4+7\rho^{1+\gamma}L^6
\leq8L^6=K.
$$
This proves \eqref{eq:field-bound} with the asserted $K$, and thus the proof is completed.

\subsection{Proof of Corollary~\ref{cor:dirichlet}}
Let $u\in C^0(\overline\Omega)$ be the unique viscosity
solution of \eqref{eq:dirichlet}.

It follows from the uniqueness and the invariance of $(\Omega,F,g)$ under the group $\mathcal G$ that $u$ takes the form
$$
 u(x)=U(z,|x^{(1)}|^2,\dots,|x^{(k)}|^2),\quad x\in \overline\Omega,
$$
for some function $U$ defined on the set of attained arguments. Then an application of Corollary \ref{cor:quadratic} yields $u\in C^{2,\beta}_{\mathrm{loc}}(\Omega)$ and, for every open set $\Omega'\Subset\Omega$, 
$$
\norm{u}_{C^{2,\beta}(\Omega')}\le C\norm{u}_{L^\infty(\Omega)}.
$$
The asserted right hand side follows from a standard barrier argument by using 
$$
 b_\pm(x):=\pm\bigl(\norm{g}_\infty+A(R_0^2-|x-x_0|^2)\bigr),\quad x\in \Omega,
$$
where $x_0\in \Omega$, $R_0:=\operatorname{diam}\Omega$, and $A:={|F(0)|}/{(2n\lambda)}$.

\section{Proof of Theorem \ref{thm:trace-cubic} and Remark \ref{rem:parameters}}\label{sec:examples}

Our proof of Theorem \ref{thm:trace-cubic} needs the following inequalities for traceless matrices. Among them, inequality (i) is standard. Inequalities (ii) and (iii) might be known but we cannot locate them in the literature.
\begin{lemma}\label{lem-5-1}
For $m\ge 3$, let $A,B\in \Sym_0(m)$, and let $s$ be the sum of the first and second largest eigenvalues of $A$. Then the following inequalities hold.
\begin{enumerate}[label=\textup{(\roman*)}]
    \item 
 $\norm B^2\leq\frac{m-1}{m}\norm B_F^2$. 
 
\item
$\tr(AB^2)\leq s\norm B_F^2$.

\item
 $\tr(AB^2)\leq\frac{m-1}{m}s\norm B_F^2$
if $\tr(AB)=0$.
\end{enumerate}
\end{lemma}

Inequality \textup{(ii)} is optimal: equality holds for
$A=B=\diag\{ m-1,-1,\dots,-1\}$.  We do not pursue an optimal coefficient in
\textup{(iii)}, since its present form suffices for our purposes.

\begin{proof}
For \textup{(i)}, let $\mu$ be an eigenvalue of $B$ with $|\mu|=\norm B$,
and denote the others by $\mu_2,\ldots,\mu_m$.
Since $\tr B=0$, Cauchy--Schwarz gives
$$ \textstyle
 \mu^2=(\sum_{j=2}^m\mu_j )^2
 \leq(m-1)\sum_{j=2}^m\mu_j^2
 =(m-1)(\norm B_F^2-\mu^2).
$$
Then \textup{(i)} follows.
For
\textup{(ii)} and \textup{(iii)}, by an orthogonal conjugation,
we may assume $A=\diag\{ a_1,\dots,a_m\}$ with $a_1\geq\cdots\geq a_m$.
Then $s=a_1+a_2\ge 0$.
Let $d:=a_1-a_2\ge 0$ and $\theta:= a_1+(m-1)a_2\geq 0$. The nonnegativity of $\theta$ is from $\tr A=0$. Write $B=(b_{ij})$.  By \textup{(i)},
\begin{equation}\label{eq:trace-first-bound} \textstyle
 \tr(AB^2)
 \leq a_2\norm B_F^2+d \sum_{j=1}^m b_{1j}^2
 \leq\left(s-{\theta}/{m}\right)\norm B_F^2.
\end{equation}
This proves \textup{(ii)}, and also proves \textup{(iii)} when
$\theta\ge s$.

Suppose now $\theta < s$ and $\tr(AB)=0$. Then we must have $a_2 < 0$, and thus $s<d$.
Since $AB$ and $B$ are both traceless, we obtain
$$  \textstyle
  b_{11}d=\sum_{j=2}^m(a_2-a_j)b_{jj}.
$$
Since $\sum_{j=2}^m (a_2-a_j)=\theta$ and its every summand is nonnegative, Cauchy--Schwarz yields
$b_{11}^2 d^2\leq\theta^2\sum_{j=2}^m b_{jj}^2$, and therefore
$b_{11}^2\leq\sum_{j=2}^m b_{jj}^2$. This implies 
$2\sum_{j=1}^m b_{1j}^2
 \leq 
 \norm B_F^2$.
The first inequality in \eqref{eq:trace-first-bound} now gives
$\tr(AB^2) \le
s\norm B_F^2/2$. Hence, \textup{(iii)} is proved.
\end{proof}

Now we are ready to give the following:
\begin{proof}[Proof of Theorem \ref{thm:trace-cubic}]
Fix $m\geq 3$. For convenience, we write $p(X):=\tr(X^3)$, and omit all the subscripts indicating dependence on $m$. For example, we write $d=d_m$. 

We differentiate on $\Sym_0(m)$.
For $X=\iota(x)$, $U=\iota(u)$ and $V=\iota(v)$ with $x,u,v\in \R^d$, the differential of $W$ is the linear form
$$
 DW(X)[U]:=\left.\frac{d}{dt}W(X+tU)\right|_{t=0}
 =Dw(x)\cdot u.
$$
For $X\ne0$, the Hessian of $W$ is the symmetric bilinear form
$$
 D^2W(X)[U,V]
 :=\left.\frac{\partial^2}{\partial s\,\partial t}
 W(X+sU+tV)\right|_{s=t=0}
 =u^TD^2w(x)v.
$$
It defines a self-adjoint operator $H(X)$ on $\Sym_0(m)$
by $\langle H(X)U,V\rangle=D^2W(X)[U,V]$.
Its matrix in the orthonormal basis $E_1,\ldots,E_d$ is exactly $D^2w(x)$.
We use $\tr$ for the trace of an $m\times m$ matrix and
$\operatorname{Tr}$ for the trace of an operator on $\Sym_0(m)$;
in particular, $\operatorname{Tr}H(X)=\Delta w(x)$.

\medskip

\emph{Step 1. We prove that for $X, Y\in \Sym_0(m)$ with $\norm{X}_F=\norm{Y}_F=1$, 
\begin{equation}\label{eq:trace-hessian}
 \langle H(X)Y,Y\rangle=6\tr(XY^2)
 -6\tr(X^2Y)\tr(XY)
 +3p(X)(\tr(XY))^2-p(X).
\end{equation}
As a consequence, 
$\operatorname{Tr}H(X)=-(3+d)p(X)$. }

Indeed, a direct computation gives
$$
 Dp(X)[Y]=3\tr(X^2Y)\quad\text{and}\quad 
 D^2p(X)[Y,Y]=6\tr(XY^2).
$$
Writing $r=r(X):=\norm X_F$, we also have
$$
 D(r^{-1})[Y]=-r^{-3}\tr(XY)\quad\text{and}\quad 
 D^2(r^{-1})[Y,Y]=3r^{-5}(\tr(XY))^2-r^{-3}\norm Y_F^2.
$$
Then \eqref{eq:trace-hessian} follows from  the above and the product rule.

Next we compute $\operatorname{Tr}H(X)$.
Let $E_1,\dots,E_d$ be an orthonormal basis of
$\Sym_0(m)$. Using \eqref{eq:trace-hessian}, we obtain
\begin{align*}
    \textstyle
 \operatorname{Tr}H(X)
 &=\textstyle\sum_{i=1}^d\langle H(X)E_i,E_i\rangle\\  
 &=\textstyle 6\tr\left(X S\right)
   -6\tr\left(X^2\sum_{i=1}^d\tr(XE_i)E_i\right)+(3-d) p(X)  \\
&=\textstyle 6\tr\left(X S\right)
   -(3+d) p(X),
\end{align*}
where $S:=\sum_{i=1}^d E_i^2$.
Since $\tr X=0$, it remains to show that
$S$ is a scalar matrix. 
Indeed, $S$ is independent of the choice of orthonormal basis.
For every $Q\in O(m)$, the matrices $QE_iQ^T$ also form an
orthonormal basis of $\Sym_0(m)$, so $QSQ^T=S$.
Thus $S$ is scalar.

\medskip

\emph{Step 2.
For $X,Y\in\Sym_0(m)$ with $\norm{X}_F=\norm{Y}_F=1$,
we prove
$$
\kappa\lambda_{\max}(H(X)-H(Y))\geq \left|\operatorname{Tr}(H(X)-H(Y)) \right|,
$$
where $\kappa:=(m+3)(d+3)>0$, and $\lambda_{\max}$ denotes the largest eigenvalue. }

If  $\operatorname{Tr}(H(X)-H(Y)) \geq 0$, then the desired inequality is trivial. By Step 1, we may assume 
$$2 q :=p(X)-p(Y)= - (3+d)^{-1} \operatorname{Tr}(H(X)-H(Y))  >0,$$
and then the desired inequality is equivalent to
$$
(m+3)\lambda_{\max}(H(X)-H(Y))\geq 2q .
$$

Let
$$A=(X-Y)/2 \quad  \text{and}\quad  B=(X+Y)/2. $$
Then $\tr(AB)=0$ and
$\norm{A}_F^2+\norm{B}_F^2=1$.
Since orthogonal conjugation $X\mapsto OXO^T$ is an orthogonal transformation of $\Sym_0(m)$
and preserves $W$,
we may assume 
$A=\diag\{ a_1,\ldots,a_m\}$, where $a_1\geq\cdots\geq a_m$.
By $\tr A=0$ and $m\ge 3$, we have $s:=a_1+a_2\ge 0$.
By a direct computation and
Lemma \ref{lem-5-1}\textup{(ii)-(iii)}, we obtain
$$
q
=3\tr(AB^2)+\tr(A^3)
\leq s\left(
{3m^{-1}(m-1)}\norm{B}_F^2+\norm{A}_F^2
\right)
\leq {3m^{-1}(m-1)} s.
$$
This gives $3s-q\geq q/(m-1)$.

First take test matrix $U$ with entries $U_{12}=U_{21}=1/\sqrt2$
and all other entries zero.
Applying \eqref{eq:trace-hessian} (with $Y=U$ there) yields
\begin{equation}\label{eq:trace-first-test}
2^{-1} \lambda_{\max}(H(X)-H(Y))
\geq3s-q + 6(q-2s)B_{12}^2.
\end{equation}
If $q\geq2s$, the right side is at least
$3s-q\geq q/(m-1)$, which implies the desired inequality.
Hence, we may assume $0<q<2s$ in the following.

Next choose $U\in\Sym_0(m)$, whose entries
outside the left upper $2\times2$ block vanish, satisfying $\norm{U}_F=1$ and $\tr(BU)=0$.
Using $a_1+(m-1)a_2\ge0$ and
Lemma~\ref{lem-5-1}\textup{(i)}, we obtain 
\begin{gather*}
    \sqrt{2}|\tr(AU)|
\leq {a_1-a_2}
\leq {m}{(m-2)^{-1}}s
\leq 3 s, \quad \text{and}\\
4 (\tr(AU))^2
\leq 2 (m-1)^{-2}{m^2a_1^2}
\leq 2 (m-1)^{-1} {m}  \norm{A}_F^2
\leq 3 \norm{A}_F^2.
\end{gather*}
In the expansion of $\tr(B^2U)$, the terms with all $B_{ij}$
indices $i,j\in \{1,2\}$ sum to
$(B_{11}+B_{22})\tr(BU)=0$.
Therefore, 
$$ \textstyle
2\sqrt{2} |\tr(B^2U)|
\leq 2\sum_{j=3}^m(B_{1j}^2+B_{2j}^2)
\leq {\norm{B}_F^2-2B_{12}^2}.
$$
Since the left upper $2\times 2$ block of $U^2$ is $I_2/2$, we have $\tr(AU^2)=s/2$.
An application of \eqref{eq:trace-hessian}, using
$\tr(A^2U)=s\tr(AU)$ 
and dropping the nonnegative term $3q(\tr(AU))^2$, yields
\begin{equation}\label{eq:trace-second-test}
\begin{aligned}
2^{-1}\lambda_{\max}(H(X)-H(Y))
&\geq3s-q-6|\tr(AU)\tr(B^2U)|
             -6s(\tr(AU))^2\\
&\geq3s-q- 9s(2^{-1}-B_{12}^2).
\end{aligned}
\end{equation}
Here the last inequality uses
$\norm{A}_F^2+\norm{B}_F^2=1$.

Multiply \eqref{eq:trace-first-test} by $3s/2$ and
\eqref{eq:trace-second-test} by $2s-q>0$, and add them up.
The terms containing $B_{12}^2$ cancel, yielding
$$
2^{-1}\lambda_{\max}(H(X)-H(Y))
\geq
\frac{(q-s)^2+s^2/2}{7s/2-q}
\geq\frac q6
\geq\frac{q}{m+3}.
$$
The middle inequality can be verified by a direct expansion. 
Step 2 is completed.

\medskip

\emph{Step 3. Construction of the operator $F$ and completion of the proof.}

Let $\mathcal K\subset \Sym(d)$ be defined by
$$
\mathcal K:=\{H(X)\mid X\in \Sym_0(m),~ \norm{X}_F=1\}.
$$
Here $H(X)\in \Sym(d)$ is its matrix representation under $\iota$. Hence, $\mathcal K=\{D^2 w(x)\mid x\in \R^d,~ |x|=1\}$.
Define
\begin{equation}\label{eq:trace-operator}
F(M):=\min_{T\in\mathcal K}
\bigl\{\operatorname{Tr}(M-T)
+\kappa\lambda_{\max}(M-T)\bigr\},
\quad M\in\Sym(d).
\end{equation}
The minimum exists since $\mathcal K$ is compact.
Clearly, $F$ satisfies \eqref{eq:ellipticity} with $\lambda=1$ and $\Lambda=1+\kappa$. 

By Step 2, one can easily see 
 $F=0$ on $\mathcal K$. From this and the homogeneity of $w$, we obtain
$$F(D^2w)=0\quad \text{in }\R^d\setminus\{0\}\quad \text{in the pointwise sense.} $$

Now we prove that $w$ is a viscosity solution of $F(D^2 w)=0$ across $0$.
Let $\varphi\in C^2$ touch $w$ from below at $0$.
Choose $r>0$ such that 
$w\geq\varphi$ on $\overline{B_r}$ and $w(0)=\varphi(0)$.
Since $w\in C^1$, we have $D\varphi(0)=Dw(0)=0$.
Fix $\delta>0$ and a unit vector $e\in\R^d$.
For $0<\varepsilon<\delta r$, let
$$
g_\varepsilon(x):=
w(x)-\varphi(x)+\delta|x|^2-\varepsilon e\cdot x.
$$
We have $g_\varepsilon(0)=0$ and
$Dg_\varepsilon(0)=-\varepsilon e$, so the minimum of
$g_\varepsilon$ on $\overline{B_r}$ is negative.
Moreover,
$$
g_\varepsilon(x)\geq\delta|x|^2-\varepsilon|x|.
$$
Consequently, any point $y_\varepsilon$ where this minimum
is attained satisfies
$0<|y_\varepsilon|<{\varepsilon}/{\delta}<r$.
Therefore,
$D^2\varphi(y_\varepsilon)-2\delta I_d
\leq D^2w(y_\varepsilon)$.
By ellipticity and the equation away from $0$,
$$
F\bigl(D^2\varphi(y_\varepsilon)-2\delta I_d\bigr)
\leq F\bigl(D^2w(y_\varepsilon)\bigr)=0.
$$
Letting first $\varepsilon\to0$ and then $\delta\to0$,
we obtain $F(D^2\varphi(0))\leq0$.
This proves $F(D^2 w)\le 0$ near $0$ in the viscosity sense.
A similar argument yields $F(D^2 w)\ge 0$ near $0$. Thus we have proved $F(D^2 w)=0$ in $\R^d$.

Finally, we check that $w$ is not twice differentiable at $0$.
If not, then the homogeneity of $w$ gives
$2w(x)=x^TD^2w(0)x$ for all $x$.
The right side is even, while $w$ is odd. Hence, $w\equiv 0$, a contradiction. Therefore, $w$ is not twice differentiable at $0$. Theorem \ref{thm:trace-cubic} is proved.
\end{proof}

We next prove the assertions in Remark~\ref{rem:parameters}.

\begin{example}\label{ex:dependence}
Fix any $m\geq 3$ and let $n:=d_m+1$. There exist an operator
$F:\Sym(n)\to\R$ satisfying \eqref{eq:ellipticity}
and a sequence $\{u_j\}\subset C^{1,1}(\R^n)$ such that $\sup_{j\geq1}\norm{u_j}_{L^\infty(B_1)}<\infty$,
$$ \textstyle
 F(D^2u_j)=0\quad\hbox{in }\R^n
 \quad\hbox{in the viscosity sense},
$$
and the following two assertions hold. Denote 
$$
 P_j(z):=u_j(0)+Du_j(0)\cdot z
             + 2^{-1} z^TD^2u_j(0)z.
$$
\begin{enumerate}[label=\textup{(\roman*)}]
\item
There exist fixed orthonormal vectors
$\nu_1,\ldots,\nu_{n-2}\in \R^n$, radii
$\rho_j\in(0,1/8]$ with $\rho_j\downarrow 0$, and
smooth hypersurfaces $S_{1,j},\ldots,S_{n-2,j}$ through $0$
such that $\rho_j$ is admissible for $S_{i,j}$ at $0$ and
$\partial_{\nu_i}u_j=\text{constant}$ on $S_{i,j}\cap B_{\rho_j}$,
for every $1\leq i\leq n-2$.
In particular, $\delta(\nu_1,\ldots,\nu_{n-2})=1$.
However, for every $\beta>0$ and $c\in(0,1)$,
$$
 \lim_{j\to\infty}
 \sup_{0<|z|<c\rho_j}
 {|u_j(z)-P_j(z)|}\big/{|z|^{2+\beta}}=+\infty.
$$

\item
There exist smooth hypersurfaces
$T_{1,j},\ldots,T_{n-2,j}$ through $0$ and linearly independent
constant unit vectors $\mu_{1,j},\ldots,\mu_{n-2,j}\in\R^n$
such that the fixed radius $\rho_0=1/8$ is admissible for $T_{i,j}$ at $0$,
$\partial_{\mu_{i,j}}u_j=\text{constant}$ on $T_{i,j}\cap B_{\rho_0}$ for every $1\leq i\leq n-2$,
and
$\delta_j:=\delta(\mu_{1,j},\ldots,\mu_{n-2,j})\to 0^+$.
However, for every $\beta>0$ and $c\in(0,1)$,
$$
 \lim_{j\to\infty}
 \sup_{0<|z|<c\rho_0}
 {|u_j(z)-P_j(z)|}\big/{|z|^{2+\beta}}=+\infty.
$$
\end{enumerate}
\end{example}

\begin{proof}
Choose an orthonormal basis $E_1,\dots,E_{d_m}$ of
$\Sym_0(m)$ with
$$
 E_1=6^{-1/2} {\diag(2,-1,-1,0,\dots,0)}
 \quad \text{and} \quad
 E_2=2^{-1/2}{\diag(0,1,-1,0,\dots,0)}.
$$
 Let $w=w_m$ and $F_m$ be given by
Theorem \ref{thm:trace-cubic} under this basis.
Write $z=(x,y)\in\R^{d_m}\times\R$, denote the standard coordinate
vectors of $\R^{d_m}$ by $e_i$ for $i=1,\dots,d_m$, and set $e_y:=(0,1)\in\R^n$.

By the regularity and homogeneity of $w$, we can choose
$L\geq2$ such that
\begin{equation}\label{eq:W-bounds}
 \begin{aligned}
 |Dw(x)-Dw(\widehat x)|\leq 2^{-1} L|x-\widehat x|,
 \quad x,\widehat x\in\R^{d_m},\quad \text{and}\\
 \norm{D^2w(x)}+|x|\norm{D^3w(x)}\leq 2^{-1} L,\quad 
 x\ne 0.
 \end{aligned}
\end{equation}
Here $\|D^3w(x)\|$ denotes the supremum of
$|\sum_{i,j,k} w_{ijk}(x)\xi_i\eta_j\zeta_k|$
over all unit vectors $\xi,\eta,\zeta\in\R^{d_m}$.
Since $w(0)=0$ and $Dw(0)=0$, the first bound also
gives $|w(x)|\leq L|x|^2/4$ for $x\in \R^{d_m}$.

\medskip

\emph{Construction of $u_j$ and $F$.}
For each $j\ge 1$, let $r_j:=2^{-j}$ and 
$$
 u_j(x,y):=w(x+r_je_1)+2^{-1}L|x|^2+x_1 y\quad \text{for }(x,y)\in \R^n.
$$
Each $u_j$ belongs to $C^{1,1}(\R^n)$. Moreover, $|x+r_je_1|\leq2$ and $|x_1y|\leq1/2$
on $B_1$, so
$$
 \norm{u_j}_{L^\infty(B_1)}
 \leq L+ L/2+1/2\leq2L.
$$
For $M\in\Sym(n)$, let $M_{xx}$ be its upper left
$d_m\times d_m$ block and let $M_{yy}$ be its $(n,n)$ entry. Define
$$
 F(M):=F_m(M_{xx}-LI_{d_m})+M_{yy},\quad \text{for } M\in \Sym(n).
$$
This operator is independent of $j$. Clearly, $F$ has the same ellipticity constants as those of $F_m$.

We verify that $F(D^2 u_j)=0$ in $\R^n$ in the viscosity sense. 
Indeed, let $\phi$ touch $u_j$ from above at $(x_0,y_0)$.
Then
$x\to \phi(x,y_0)- 2^{-1}L|x|^2-x_1y_0$
touches $w(x+r_je_1)$ from above at $x_0$. Hence,
$F_m(D^2_{xx}\phi(x_0,y_0)-LI_{d_m})\geq0$.
Since $u_j(x_0,y)$ is affine in $y$,
$\phi_{yy}(x_0,y_0)\geq0$.
Therefore, $F(D^2\phi(x_0,y_0))\geq 0$. This proves $F(D^2 u_j)\ge 0$ in $\R^n$.
A similar argument yields $F(D^2 u_j)\le 0$, and thus $F(D^2 u_j)=0$ in $\R^n$. 

We next derive a remainder estimate used in proving (i) and (ii). The chosen basis gives
$$
 w(e_1+te_2)=h(t^2),\quad t\in \R,
$$
where 
$$
 h(s):=\frac{1-3s}{\sqrt{6(1+s)}},\quad s\ge 0. $$
Since
$h''(s)>0$ for $s\geq 0$,
we have
$$  \textstyle
 h(t^2)-h(0)-h'(0)t^2
 =\int_0^{t^2}(t^2-s)h''(s)\,ds>0
 \quad \text{for }t>0.
$$
Since $u_j$ and $w(x+r_je_1)$ differ by a quadratic, the homogeneity of $w$ yields that, for every
fixed $t>0$ and $\beta>0$,
\begin{equation}\label{eq:error-scaling}
 \frac{|u_j(r_jte_2,0)-P_j(r_jte_2,0)|}
 {|(r_jte_2,0)|^{2+\beta}}
 =\frac{h(t^2)-h(0)-h'(0)t^2}{t^{2+\beta}}r_j^{-\beta}
 \to +\infty.
\end{equation}

\medskip

\emph{Proof of \textup{(i)}.}
Let
$\rho_j:={r_j}/{8}\downarrow 0$,
 and let
 $$
 \nu_i:=(e_i,0)\quad\text{for }1\leq i\leq n-3,\quad \text{and}
 \quad \nu_{n-2}:=e_y.
$$
Since $\{\nu_i\}$ are orthonormal, we have $\delta(\nu_1,\dots,\nu_{n-2})=1$.

For $1\leq i\leq n-3$, define
 \begin{align*}
 g_{i,j}(x,y)
 &:=\partial_{x_i}u_j(x,y)-\partial_{x_i}u_j(0)\\
 &=\partial_iw(x+r_je_1)-\partial_iw(r_je_1)
   +Lx_i+\delta_{i1}y,
 \end{align*}
where $\delta_{i1}$ is the Kronecker symbol. Let
$$
 S_{i,j}:=\{g_{i,j}=0\}\cap B_{2\rho_j}\quad \text{for }1\le i\le n-3,\quad \text{and}
 \quad
 S_{n-2,j}:=\{x_1=0\}\cap B_{2\rho_j}.
$$
Clearly, $0\in S_{i,j}$ for every $i$. 
On $B_{2\rho_j}$, we have $|x+r_je_1|\geq r_j/2$.
Thus $g_{i,j}$ is smooth there, and \eqref{eq:W-bounds} gives
\[
 \partial_{x_i}g_{i,j}
 =\partial_{ii}w(x+r_je_1)+L\ge  L/2 \quad \text{and}
 \quad
 \norm{D^2g_{i,j}}\leq  L/{r_j}\quad \text{on }B_{2\rho_j}.
\]
The implicit function theorem shows that $S_{i,j}$ is a
smooth hypersurface.
Recall that the second fundamental form of a regular level set $\{g=0\}$,
is, up to sign, the restriction
of $D^2g/|Dg|$ to tangent vectors. From the above, we obtain $ \norm{\II_{S_{i,j}}}\leq {2}/{r_j}$ and thus
$$ 
 \rho_j\sup_{S_{i,j}\cap B_{\rho_j}}
 \norm{\II_{S_{i,j}}}\leq 1/4\quad \text{for }1\le i\le n-3.
$$
This also holds for $i=n-2$, since $S_{n-2,j}$ is flat. By continuity, $S_{i,j}\cap B_{\rho_j}$ is relatively closed in
$B_{\rho_j}$.
Hence, $\rho_j$ is admissible for $S_{i,j}$ at $0$. 

From the definition of $g_{i,j}$ and the identity
$\partial_yu_j=x_1$, we have $\partial u_j /\partial \nu_i=\text{constant}$ on $S_{i,j}$ for every $i$.
For any fixed $\beta>0$ and $c\in(0,1)$, choose $t=c/16$.
Then $|(r_jte_2,0)|=c\rho_j/2$, and
\eqref{eq:error-scaling} gives the asserted divergence in (i). Hence, (i) is proved.

\medskip

\emph{Proof of \textup{(ii)}.}
Let $\rho_0:=1/8$, and let
$$
 \mu_{i,j}:=\frac{\tau_j(e_i,0)+e_y}{\sqrt{1+\tau_j^2}}
 \quad\text{for }1\leq i\leq n-3,\quad \text{and}
 \quad \mu_{n-2,j}:=e_y.
$$
where $ \tau_j:={r_j^2}/{(2L)}\downarrow 0$. Clearly, $\mu_{1,j},\dots, \mu_{n-2,j}$ are linearly independent unit vectors.
Moreover, the definition of the
independence number $\delta_j$ gives
$$
 0<\delta_j\leq
 {|\mu_{1,j}-e_y|}/{\sqrt2}\to 0.
$$
Here we used the coefficients $1/\sqrt2$ and $-1/\sqrt2$
for the first and last directions, with all others zero.

Let
$$
 T_{i,j}:=\{f_{i,j}=0\}\cap B_{2\rho_0}\quad \text{for }1\le i\le n-3,\quad \text{and}  \quad T_{n-2,j}:=\{x_1=0\}\cap B_{2\rho_0},
$$
where  $f_{i,j}(x,y):=x_1+\tau_j g_{i,j}(x,y)$ and $g_{i,j}$ is defined as in the proof of (i).
Clearly, $0\in T_{i,j}$ for every $i$.
In the following, we prove that each $T_{i,j}$ is a smooth hypersurface.
We first check that the zero level set
of $f_{i,j}$ avoids the singularity of $w(x+r_je_1)$.
By \eqref{eq:W-bounds} and $L\geq 2$, we have
$$
 |g_{i,j}(x,y)|
 \leq 2^{-1} L|x|+L|x_i|+|y|\leq2L|(x,y)|.
$$
At every point of $T_{i,j}$, the above gives
$$
 |x_1|=\tau_j|g_{i,j}(x,y)|
 \leq r_j^2|(x,y)|< {r_j^2}/{4},
$$
and hence $|x+r_je_1|\geq x_1+r_j\geq r_j/2$.
Therefore, $f_{i,j}$ is smooth near each point of $T_{i,j}$.
At these points, as in the proof of (i), we have $|Dg_{i,j}|\leq 2L$ and
$\norm{D^2g_{i,j}}\leq L/{r_j}$.
Therefore, 
$$
 |Df_{i,j}|
 \geq1-r_j^2\geq 1/2 \quad \text{and}
 \quad
 \norm{D^2f_{i,j}}\leq {r_j}/{2}\quad \text{near }T_{i,j}.
$$
The implicit function theorem shows that $T_{i,j}$ is a smooth hypersurface.
As in the proof of (i), we have
$\norm{\II_{T_{i,j}}}\leq r_j$, and therefore,
$$
 \rho_0\sup_{T_{i,j}\cap B_{\rho_0}}
 \norm{\II_{T_{i,j}}}\leq\rho_0r_j\leq 1 \quad \text{for }1\le i\le n-3.
$$
This also holds for $i=n-2$, since $T_{n-2,j}$ is flat.
By continuity, $T_{i,j}\cap B_{\rho_0}$ is relatively closed in
$B_{\rho_0}$. Hence, $\rho_0$ is admissible for $T_{i,j}$ at $0$.

A computation yields
$$
 \partial_{\mu_{i,j}}u_j-\partial_{\mu_{i,j}}u_j(0)
 ={f_{i,j}}{({1+\tau_j^2})^{-1/2}}\quad \text{for }1\le i\le n-3.
$$
This and the identity
$\partial_yu_j=x_1$ imply $\partial u_j/\partial \mu_{i,j}=\text{constant}$ on $T_{i,j}$ for every $i$.
Finally, since $\rho_j\leq\rho_0$, the asserted divergence in \textup{(ii)} follows from that in \textup{(i)}.
\end{proof}

\appendix

\section{\texorpdfstring
  {The cubics $\tr(X^3)$ on $\Sym_0(m)$}
  {The cubics tr(X\textasciicircum3) on Sym\textunderscore0(m)}}
\label{app-sec-realizations}

We give explicit isometries $\iota_m: \R^{d_m} \to \Sym_0(m)$ to identify $\tr(X^3)$ with the Cartan cubic $P(x)$ when $m=3$, and to represent the Hsiang cubic by $D(x)$ when $m=4$. The choices of $\iota_m$ are not unique; any other
choices differ by an orthogonal change of variables.

\subsection{The Cartan cubic}
For $x\in\R^5$, take
$$
 \iota_3(x)=\frac1{\sqrt6}
 \begin{pmatrix}
 2x_1&\sqrt3x_3&\sqrt3x_4\\
 \sqrt3x_3&-x_1+\sqrt3x_2&\sqrt3x_5\\
 \sqrt3x_4&\sqrt3x_5&-x_1-\sqrt3x_2
 \end{pmatrix}.
$$
Then $\iota_3$ is a linear isometry onto $\Sym_0(3)$.

We next compute $\tr(\iota_3(x)^3)$. 
A direct computation gives
$\det (\sqrt{6} \iota_3(x))
 =2P(x)$.
For a traceless $3\times3$ symmetric matrix $X$, we have
$\tr(X^3)
=3\det X$. Therefore,
$$ \sqrt{6}\tr(\iota_3(x)^3)=P(x).
$$
It follows that $\sqrt{6} w_3(x)=P(x)/|x|$.

\subsection{The Hsiang cubic}
In \cite[Example 1]{Hsiang1967},
the Hsiang cubic is originally defined as the coefficient $\sigma_3(X)$
of $t$ in $-\det(X-tI_4)$ for $X\in\Sym_0(4)$. Let
$\lambda_1,\ldots,\lambda_4$ be the eigenvalues of $X$. Since their
sum is zero, Newton's identity yields
$$ \textstyle \sigma_3(X):=
 \sum_{i<j<k}\lambda_i\lambda_j\lambda_k
 = 3^{-1}\sum_i\lambda_i^3= 3^{-1}\tr(X^3).
$$

For $x\in\R^9$, take
$$
 \iota_4(x)=\frac12
 \begin{pmatrix}
 x_1+x_5+x_9&x_6-x_8&x_7-x_3&x_2-x_4\\
 x_6-x_8&x_1-x_5-x_9&x_2+x_4&x_3+x_7\\
 x_7-x_3&x_2+x_4&-x_1+x_5-x_9&x_6+x_8\\
 x_2-x_4&x_3+x_7&x_6+x_8&-x_1-x_5+x_9
 \end{pmatrix}.
$$
Then $\iota_4$ is a linear
isometry onto $\Sym_0(4)$.

We next compute $\tr(\iota_4(x)^3)$.  
For every symmetric $4\times4$ matrix $K$,
$$ \textstyle
 \tr(K^3)=\sum_i K_{ii}^3
  +3\sum_{i<j}(K_{ii}+K_{jj})K_{ij}^2
  +6\sum_{i<j<k}K_{ij}K_{jk}K_{ki}.
$$
For $K=2\iota_4(x)$, the three terms on the right hand side are
 \begin{gather*} \textstyle
 \sum_iK_{ii}^3=24x_1x_5x_9,\\ \textstyle
 3\sum_{i<j}(K_{ii}+K_{jj})K_{ij}^2
 =-24(x_1x_6x_8+x_3x_5x_7+x_2x_4x_9), \quad \text{and}   \\  \textstyle
 6\sum_{i<j<k}K_{ij}K_{jk}K_{ki}
 =24(x_2x_6x_7+x_3x_4x_8).
 \end{gather*}
Adding them yields
$$
 \tr(\iota_4(x)^3)=3D(x).
$$
From this, the representation of the Hsiang cubic is exactly $D(x)$ in the above
coordinates.
Also, it follows that $w_4(x)=3 D(x)/|x|$.


\begin{thebibliography}{99}
\bibitem{ASC}
S. Armstrong, L. Silvestre, and C. K. Smart,
\emph{Partial regularity of solutions of fully nonlinear, uniformly elliptic equations},
Comm. Pure Appl. Math. \textbf{65} (2012), 1169--1184.


\bibitem{CC-jmpa}
X. Cabr\'e and L. A. Caffarelli,
\emph{Interior $C^{2,\alpha}$ regularity theory for a class of
nonconvex fully nonlinear elliptic equations},
J. Math. Pures Appl. \textbf{82} (2003), 573--612.

\bibitem{Caffarelli-89}
L. A. Caffarelli,
\emph{Interior a priori estimates for solutions of fully nonlinear equations},
Ann. of Math. (2) \textbf{130} (1989), 189--213.


\bibitem{Caffarelli1995FullyNE}
L. A. Caffarelli and X. Cabr\'e,
\emph{Fully Nonlinear Elliptic Equations},
Amer. Math. Soc. Colloq. Publ., vol.~43,
American Mathematical Society, Providence, RI, 1995.





\bibitem{CS}
L. A. Caffarelli and L. Silvestre,
\emph{On the {E}vans-{K}rylov theorem},
Proc. Amer. Math. Soc. \textbf{138},
(2010), 263--265.


\bibitem{CY}
L. A. Caffarelli and Y. Yuan,
\emph{A priori estimates for solutions of fully nonlinear equations with convex level set},
Indiana Univ. Math. J. \textbf{49},
(2000), 681--695. 

\bibitem{Cartan}
E. Cartan,
\emph{Familles de surfaces isoparam\'etriques dans les espaces \`a{} courbure constante},
Ann. Mat. Pura Appl. \textbf{17},
(1938), 177--191. 


\bibitem{Chu}
B.Z. Chu,
\emph{On singular sets of fully nonlinear uniformly elliptic equations},
\newblock \href{https://arxiv.org/abs/2608.18217v2}{arXiv:2608.18217v2} [math.AP]. 

\bibitem{CIL}
M. G. Crandall, H. Ishii, and P.-L. Lions,
\emph{User's guide to viscosity solutions of second order partial
differential equations},
Bull. Amer. Math. Soc. (N.S.) \textbf{27} (1992), 1--67.


\bibitem{Evans-82}
L. C. Evans,
\emph{Classical solutions of fully nonlinear, convex, second-order elliptic equations}, Comm. Pure Appl. Math. \textbf{35} (1982), 333--363. 

\bibitem{GT}
D. Gilbarg and N. S. Trudinger,
\emph{Elliptic partial differential equations of second order},
2nd ed., Springer-Verlag, Berlin, 1983.


\bibitem{Hsiang1967}
W.-Y. Hsiang,
\emph{Remarks on closed minimal submanifolds in the standard
Riemannian $m$-sphere},
J. Differential Geom. \textbf{1} (1967), 257--267.


\bibitem{Ishii-89}
H. Ishii,
\emph{On uniqueness and existence of viscosity solutions of fully nonlinear second-order elliptic {PDE}s}, Comm. Pure Appl. Math. \textbf{42} (1989), 
15--45. 


\bibitem{Jensen-88}
R. Jensen,
\emph{The maximum principle for viscosity solutions of fully nonlinear second order partial differential equations}, Arch. Rational Mech. Anal. \textbf{101} (1988), 
1--27. 



\bibitem{Kazdan}
J. L. Kazdan,
\emph{Prescribing the curvature of a {R}iemannian manifold},
CBMS Regional Conference Series in Mathematics 57, 1985.

\bibitem{Krylov-82}
N. V. Krylov,
\emph{Boundedly nonhomogeneous elliptic and parabolic equations},
Izv. Akad. Nauk SSSR Ser. Mat. \textbf{46} (1982), 487--523; English transl. in Math. USSR Izv. \textbf{20} (1983), 459--492.


\bibitem{Krylov-83}
{\sc ---}
\emph{Boundedly nonhomogeneous elliptic and parabolic equations in a domain},
Izv. Akad. Nauk SSSR Ser. Mat. \textbf{47} (1983), 75--108; English transl. in Math. USSR Izv. \textbf{22} (1984), 67--97.




\bibitem{NV}
N. Nadirashvili and S. Vl\u{a}du\c{t},
\emph{On axially symmetric solutions of fully nonlinear elliptic equations},
Math. Z. \textbf{270} (2012), 331--336.



\bibitem{NTV}
N.~Nadirashvili, V.~Tkachev, and S.~Vl\u{a}du\c{t},
\emph{A non-classical solution to a Hessian equation from Cartan
isoparametric cubic},
Adv. Math. \textbf{231} (2012), 1589--1597.

\bibitem{NTV-book}
{\sc ---}
\emph{Nonlinear elliptic equations and nonassociative algebras},
Math. Surveys Monogr. 200, American Mathematical Society, Providence, RI, 2014.

\bibitem{Nirenberg}
L. Nirenberg,
\emph{On nonlinear elliptic partial differential equations and
H\"older continuity},
Comm. Pure Appl. Math. \textbf{6} (1953), 103--156.

\bibitem{S}
L. Silvestre,
\emph{Singular solutions to parabolic equations in nondivergence
form},
Ann. Sc. Norm. Super. Pisa Cl. Sci. \textbf{23} (2022), 993--1011.



\bibitem{SS}
L. Silvestre and B. Sirakov,
\emph{Boundary regularity for viscosity solutions of fully nonlinear
elliptic equations},
Comm. Partial Differential Equations \textbf{39} (2014), 1694--1717.


\bibitem{Tkachev2010}
V. Tkachev,
\emph{Minimal cubic cones via Clifford algebras},
Complex Anal. Oper. Theory \textbf{4} (2010), 685--700.

\bibitem{Tkachev2014}
{\sc ---}
\emph{A Jordan algebra approach to the cubic eiconal equation},
J. Algebra \textbf{419} (2014), 34--51.


\bibitem{Tkachev2019}
{\sc ---}
\emph{Spectral properties of nonassociative algebras and breaking
regularity for nonlinear elliptic type PDEs},
St. Petersburg Math. J. \textbf{31} (2020), 223--240.


\bibitem{Trudinger-89}
N. S. Trudinger,
\emph{On regularity and existence of viscosity solutions of nonlinear second order elliptic equations}, Progr. Nonlinear Differential Equations Appl. 2, 939--957, Birkh\"auser, Boston, MA, 1989.

\bibitem{Trudinger-twice-differentiability}
{\sc ---}
\emph{On the twice differentiability of viscosity solutions of nonlinear elliptic equations},
Bull. Austral. Math. Soc. \textbf{39} (1989), 443--447.

\end{thebibliography}
\end{document}